\documentclass{amsart}
\usepackage{amsfonts}
\usepackage{amsfonts,amsmath,amssymb}
\usepackage{mathrsfs,mathtools,stmaryrd,wasysym}
\usepackage{enumerate}
\usepackage{xspace,./mydef} 
\usepackage{hyperref}
\usepackage{xcolor}
\usepackage{accents}
\usepackage{tcolorbox}
\usepackage[numbers]{natbib}

\usepackage{tikz}
\usepackage{tkz-euclide}
\usetkzobj{all}
\usetikzlibrary{decorations.pathreplacing,shapes.misc,calc,intersections,arrows,external}

\usepackage{pdfsync}

\newcommand{\TheTitle}{A priori error estimates of regularized elliptic problems}
\newcommand{\ShortTitle}{A priori error estimates of regularized elliptic problems}

\begin{document}

\title[\ShortTitle]{{\TheTitle}}

\author{Luca Heltai}
\address{SISSA - International School for Advanced Studies, Via Bonomea 265, 34136 Trieste, Italy}
\email{luca.heltai@sissa.it}
\author{Wenyu Lei}
\address{SISSA - International School for Advanced Studies, Via Bonomea 265, 34136 Trieste, Italy}
\email{wenyu.lei@sissa.it}

\date{Draft version of \today.}

\begin{abstract}
  Approximations of the Dirac delta distribution are commonly used to
  create sequences of smooth functions approximating nonsmooth
  (generalized) functions, via convolution. In this work we show
  a-priori rates of convergence of this approximation process in
  standard Sobolev norms, with minimal regularity assumptions on the
  approximation of the Dirac delta distribution. The application of
  these estimates to the numerical solution of elliptic problems with
  singularly supported forcing terms allows us to provide sharp $H^1$
  and $L^2$ error estimates for the corresponding regularized
  problem. As an application, we show how finite element
  approximations of a regularized immersed interface method result in
  the same rates of convergence of its non-regularized counterpart,
  provided that the support of the Dirac delta approximation is set to
  a multiple of the mesh size, at a fraction of the implementation
  complexity. Numerical experiments are provided to support our
  theories.
\end{abstract}

\maketitle



\section{Introduction}\label{s:intro}

Singular source terms are often used in partial differential equations
(PDE) to model interface problems, phase transitions, or
fluid-structure interaction problems. The immersed boundary method
(IBM, ~\cite{Peskin-2002-a}) is a good example of a model problem
where a Dirac delta distribution supported on an immersed fiber or
surface is used to capture complex dynamics that are happening on
possibly moving interfaces. Similar forcing terms can be used, for
example, to model fictitious boundaries in the
domain~\citep{GlowinskiPanPeriaux-1994-a}, or to couple solids and
fluids across non-matching
discretizations~\citep{HeltaiCostanzo-2012-a}.

When the discretization scheme is based on variational principles,
singular sources may be incorporated exactly in the numerical scheme
by the definition of their action on test
functions~\cite{BoffiGastaldi-2003-a, BoffiGastaldiHeltai-2008-a,
  HeltaiCostanzo-2012-a}. However, if the discretization scheme is
based on finite differences~\cite{Peskin-2002-a} or finite
volumes~\cite{MittalIaccarino-2005-a}, this procedure is cumbersome,
and a typical alternative is to regularize the source term by taking
its convolution with an approximation of the Dirac delta
distribution, or by modifying the differential operators themselves to
incorporate the knowledge about the interface~\cite{LevequeLi-1994-a}.

Several works are dedicated to the design of good Dirac
approximations~\cite{MR3429589,TornbergEngquist-2004-a} to use in this
regularization process, and their convergence properties are well
known in the literature of immersed boundary
methods~\cite{LaiPeskin-2000-a,GriffithPeskin-2005-a}.

When the source term is a single Dirac distribution, \cite{MR3429589}
proved convergence in both the weak-$\ast$ topology of distributions,
as well as in a weighted Sobolev norm, provided that certain momentum
conditions are satisfied.

In many applications, however, the Dirac distribution is not used
directly as a source term, but only in formal convolutions to
represent the coupling between overlapping domains. In this context,
the resulting source term may or may not be singular, according to the
co-dimension of the immersed domain itself.  To fix the ideas,
consider a source term of the kind
\begin{equation*}
  F(x) := \int_B \delta(x-y) f(y) \diff y, \qquad x \in \Omega
  \subseteq \Rd, \qquad B \subset \Omega.
\end{equation*}
The above formalism is used in immersed methods to represent a
(possibly singular) coupling between $B$ and $\Omega$. Here $\delta$
is a $d$-dimensional Dirac distribution. If the co-dimension of the
immersed domain $B$ is zero, then the above forcing term reduces to
\begin{equation*}
  F(x) = \chi_B(x)f(x),
\end{equation*}
where $\chi_B$ is the indicator function of $B$, owing to the
distributional definition of $\delta$. However, if the co-dimension of
$B$ is greater than zero, the integration over $B$ does not exhaust
the singularity of the Dirac distribution: the resulting $F$ is still
singular, and it should be interpreted as the distribution whose
effect on smooth functions $\varphi$ is given by:
\begin{equation*}
  \langle F, \varphi\rangle := \int_B f(y) \varphi(y) \diff y, \qquad \forall \varphi \in C^\infty_c(\overline\Omega).
\end{equation*}
In the three dimensional case, assuming that $f\in L^2(B)$, the regularity of $F$ goes from being 
$L^2(\Omega)$ when the co-dimension of $B$ is zero, to a negative 
Sobolev space which cannot be smoother than $H^{-1/2}(\Omega)$,
$H^{-1}(\Omega)$, and $H^{-3/2}(\Omega)$ when $B$ is a Lipschitz surface,
Lipschitz curve, or point, respectively.

In all cases, a regularization of $F$ is possible by convolution with
an approximate (possibly smooth) Dirac function. In most of the
literature that exploits this technique from the numerical point of
view, pointwise convergence, truncation, and Taylor expansions are
used to argue that high order convergence can be achieved in $L^p(\Omega)$
norms for the numerical approximation of the regularized problem,
provided that some specific conditions are met in the construction of
the approximate Dirac
function~\cite{Mori-2008-a,LiuMori-2012-a,LiuMori-2014-a}. Convergence
of the regularized solution to the exact solution of the original
problem is usually left aside, with the notable exceptions of the work
by~\cite{MR3429589}, where only the degenerate case of $B$ being a
single point is considered, and the work
by~\cite{SaitoSugitani-2019-a}, where suboptimal estimates in
$W^{-1, p}(\Omega)$ norm for some $p\in[1,\tfrac d{d-1})$ are provided 
for finite element discretizations of
the immersed boundary method applied to the Stokes problem.

Even though the mollification (also known as regularization) technique
is widely used in the mathematical analysis community, very little is
known about the the speed at which regularized functions converge to
their non-regularized counterparts in standard Sobolev norms, when
non-smooth approximations of the Dirac delta distribution are considered.

In this work we provide convergence results in standard Sobolev norms
that mimic closely the a-priori estimates available in finite element
analysis, and we investigate the performance of the finite element
method using regularized forcing data, in the energy and
$L^2(\Omega)$ norms, for a class of singular source terms commonly
used in interface problems, fluid structure interaction problems, and
fictitious boundary methods.

We start by providing an {a priori} framework in standard Sobolev spaces
for unbounded domains, to study the convergence of regularized
functions to their non-regularized counterparts, with minimum
requirements on the regularization kernel.

We extend this framework to bounded domains, assuming that the support
of the forcing data is away from the physical domain (see the
definition in \eqref{i:interface-assumption}), and we study
the convergence speed of elliptic problems with regularized forcing
terms to their non-regularized counterparts, with minimum requirements
on the support of the forcing term. For compactly supported kernels,
in Theorem~\ref{t:h1-reg}, we provide sharp convergence estimates in
the energy norm in terms of powers of the radius of the support. We
also provide an $L^2(\Omega)$ error estimate by following a duality
argument (or Aubin-Nitsche trick) and using the $H^2$ interior
regularity of a dual problem; we refer to Theorem~\ref{t:reg-estimate}
for more details. We note that although we only consider Dirichlet
boundary conditions in this paper, the convergence results that we
derive for the regularized problem can be also applied to elliptic
problems with other boundary conditions; see Remark~\ref{r:bc}.

As an application, we investigate how the regularization affects the
total error of the finite element approximation of an interface
problem via immersed methods, and we show that all the estimates we
obtain are sharp. Even though a regularization in this case is not
necessary when using the finite element
method~\cite{BoffiGastaldi-2003-a}, the biggest advantage of using
regularization comes from the fact that its numerical implementation
is trivial; see
Remark~\ref{r:compute}. Theorem~\ref{t:discrete-h1-estimate} and
\ref{t:discrete-l2-estimate} show that the regularization does not
affect the overall performance of the finite element approximation in
both energy and $L^2(\Omega)$ norms when choosing the regularization
parameter to be a multiple of the maximal size of the quasi-uniform
subdivision of $\Omega$.

The rest of the paper is organized as follows. We first introduce some
notations in Section~\ref{s:prelim}, and provide general results in
both bounded and unbouned domains in
Section~\ref{s:regularization}. These results are applied to a general
model elliptic problem in Section~\ref{ss:model-problem}, where we use
the results of Section~\ref{s:regularization} to derive sharp error
estimates in energy and $L^2(\Omega)$ norms. An application to
immersed methods is provided in Section~\ref{s:ib}, where we apply the
theory to a finite element approximation of an interface problem.
Finally we discuss the numerical implementation of the singular
forcing data (with and without regularization) and validate our
findings via numerical simulations in Section~\ref{s:numerical}.

\section{Notations and preliminaries}\label{s:prelim}
In what follows, we set $\Omega$ to be a bounded Lipschitz domain in $\Rd$
with $d=2$ or $3$. We write $A\preceq B$ when $A\le cB$ for some constant $c$
independent of regularization and discretization parameters (mentioned in
Section~\ref{s:intro}) in $A$ and $B$. We indicate $A\sim B$ when both
$A\preceq B$, and $B\preceq A$. For $a\in\Real$, we denote with $a^-$
(or $a^+$) any real number strictly smaller (or greater) than $a$.

For $x\in \Rd$, we use $|x|$ to indicate the euclidean norm of $x$,
and given a normed space $X$, we denote by $X'$ and
$\langle\cdot,\cdot\rangle_{X',X}$ its dual space and the duality
pairing, respectively. We also denote by $\|\cdot\|_X$ and
$\|\cdot\|_{X'}$ the norm of $X$ and the operator norm of $X'$, \ie
\[
	\|v\|_{X'}:=\sup_{v\in X, \|v\|_X\neq 0}\frac{\langle v,w\rangle_{X,X'}}{\|w\|_{X}} .
\]

\subsection{Sobolev spaces}

For $s\in\mathbb N$ and $p>1$, we denote by $W^{s,p}(\Omega)$,
$H^s(\Omega)$ and $L^2(\Omega)$ the usual Sobolev spaces. For
convention we set $H^0(\Omega)=L^2(\Omega)$ and denote with
$(\cdot,\cdot)_\Omega$ the $L^2(\Omega)$ inner product. 
{We denote with $|\cdot|_{H^s(\Omega)}$ the usual semi-norm
of $H^s(\Omega)$, and we set $|\cdot|_{H^0(\Omega)}:=\|\cdot\|_{L^2(\Omega)}$.}
Let $\Hunz$ be
the set of functions in $H^1(\Omega)$ vanishing on
$\partial\Omega$. It is well known that $\Hunz$ is the closure of the
space of infinitely differentiable functions with compact support in
$\overline\Omega$ (denoted by $C^\infty_c(\overline\Omega)$ or
$\mathcal D(\Omega)$) with respect to the norm of $H^1(\Omega)$. For
$s\in(0,1)$, $H^s(\Omega)$ denotes the space of functions whose
Sobolev-Slobodeckij norm
\[
	\|v\|_{H^s(\Omega)} :=\Big( \|v\|_{L^2(\Omega)}^2 + \int_\Omega\int_\Omega \frac{(v(x)-v(y))^2}{|x-y|^{d+2s}}\diff x\diff y\Big)^{1/2}
\]
is finite. Similarly, for $s\in(1,2)$, the norm of $H^s(\Omega)$ is 
\[
	\|v\|_{H^s(\Omega)}=\big(\|v\|_{L^2(\Omega)}^2+\|\GRAD v\|_{H^{s-1}(\Omega)}^2\big)^{1/2} .
\]
It is well known that for $s\in(0,2)$,
$H^s(\Omega)=[L^2(\Omega),H^2(\Omega)]_s$, where $[X,Y]_s$ denotes the
interpolation space between $X$ and $Y$ using the real method. For
$s\in[0,2]$, we denote $H^{-s}(\Omega)=H^s(\Omega)'$.

\subsection{Distributions}
We indicate with $\mathcal D(\Rd)$ or $\mathcal D(\Omega)$ the spaces
of smooth functions with compact support in $\Rd$ or in $\Omega$, i.e.,
$\mathcal D(\Rd) := C^{\infty}_c(\Rd)$ and
$\mathcal D(\Omega)=C^{\infty}_c(\overline{\Omega})$, and with $\mathcal D'(\Rd)$
and $\mathcal D'(\Omega)$ their dual spaces (also known as the spaces
of distributions). The $d$-dimensional Dirac delta distribution
centered in $x$ is
defined as the linear functional $\delta_x$ in $\mathcal D'$ such that
\begin{equation*}
  \langle\delta_x, v\rangle_{\mathcal D'(A), \mathcal D(A)} =\int_A
  \delta(x-y) v(y) \diff y := v(x) \quad \forall v
  \in \mathcal D(A), \quad \forall x\in A,
\end{equation*}
where $A$ is either $\Rd$ or $\Omega$.

\section{Regularization}\label{s:regularization}

In the mathematical literature, mollifiers were introduced by
Sobolev~\cite{soboleff1938theoreme}, and formalized by
Friedrichs~\cite{Friedrichs1944}, to provide an effective way of
regularizing non smooth or singular functions by performing the
convolution with a smooth and compactly supported function that
integrates to one (the kernel, or mollifier).

The outcome of the procedure is a function that has the same
regularity property of the kernel, normally taken to be $C^\infty$.
In this section we relax the requirements of the kernel, and we
provide some basic results that can be obtained by regularizing with
approximated Dirac distributions that are not necessarily
$C^\infty$. In particular, we consider functions $\psi$ that satisfy
the following assumptions:
\begin{assumption}\label{a:app-dirac}
  Given $k \in \mathbb{N}$, let $\psi(x)$ in $L^\infty(\Rd)$ be
  such that
\begin{enumerate}[1.]
\item \textbf{Compact support:} \\
  $\psi(x)$ is compactly supported, with support $\emph{supp}(\psi)$
  contained in $B_{r_0}$ (the ball centered in zero with radius $r_0$)
  for some $r_0>0$;
\item \textbf{$k$-th order moments condition}:\\
  \begin{equation}
    \int_\Rd y^\alpha_i\psi(x-y)\diff y = x^\alpha_i \qquad
    i=1\ldots d, \quad 0\leq\alpha\leq k, \quad\forall x \in \Rd;
\end{equation}
\end{enumerate}
\end{assumption}

\begin{lemma}[Convergence to $\delta$] A function $\psi$ that
  satisfies Assumption~\ref{a:app-dirac} for some $k\geq 0$,
  defines a one-parameter family of Dirac delta approximations, i.e.,
  define
  \begin{equation}
    \label{e:dirac-approximation}
    \depsilon := \frac{1}{\varepsilon^d}
    \psi\left(\frac{x}{\varepsilon}\right)
  \end{equation}
  then
  \begin{equation*}
    \lim_{\varepsilon \to 0} \depsilon(x) = \lim_{\varepsilon \to 0}
    \frac{1}{\varepsilon^d} \psi\left(\frac{x}{\varepsilon}\right) = \delta(x),
  \end{equation*}
  where $\delta (x)$ is the Dirac delta distribution and the limit
  must be understood in the space of Schwartz distributions.
\end{lemma}

Notice that Assumption~\ref{a:app-dirac}.2 with $k=0$ implies that
both $\psi$ and $\depsilon$ have unit integral. Moreover, we make no
additional assumptions on the global regularity of $\psi$, except from
requiring it to be in $L^\infty(\Rd)$. 

Any function $\psi$ that satisfies Assuption~\ref{a:app-dirac} defines a
one-parameter family of Dirac delta approximations $\depsilon$,
through which it is possible to define a regularization in both $\Rd$
and $\Omega$:

\begin{definition}[Regularization]
  \label{d:regularization}
  For a function $v\in L^1(A)$\, we define its regularization
  $\vepsilon(x)$ in the domain $A$ (either $\Omega$ or $\Rd$) {through}
  the mollifier $\psi$ by
\begin{equation}\label{e:regularization-A}
  \vepsilon(x) := \int_{A} \depsilon(x-y) v(y) \diff y,\qquad\forall x\in A,
\end{equation}
where $\depsilon$ is defined as in
Eq.~\eqref{e:dirac-approximation}, i.e.,
\begin{equation*}
  \depsilon := \frac{1}{\varepsilon^d} \psi\left(\frac{x}{\varepsilon}\right),
\end{equation*}
and $\psi$ satisfies Assumption~\ref{a:app-dirac} for some $k\geq 0$.

For functionals $F$ in negative Sobolev spaces, say $F\in H^{-s}(A)$,
with $s\geq 0$, we define its regularization $F^\varepsilon$ by the
action of $F$ on $v^\varepsilon$ with $v\in H^s(A)$, \ie
\begin{equation*}
  \langle F^\varepsilon, v \rangle_{H^{-s}(A), H^s(A)} :=   \langle
  F, v^\varepsilon \rangle_{H^{-s}(A), H^s(A)} .
\end{equation*}
\end{definition}

\begin{lemma}[$L^1$ growth control]
  \label{l:L1-control}
  A Dirac approximation $\depsilon$ constructed from a function $\psi$
  that satisfies Assumptions~\ref{a:app-dirac} (irrespective of
  $k\geq 0$), also satisfies the following polynomial growth
  condition:
  \begin{equation}
    \label{e:L1-control}
    \||x|^m\depsilon(x)\|_{L^1(\Rd)} 
    \preceq \varepsilon^m,  \qquad 0\leq m \in \Real
  \end{equation}
  where the hidden constant depends on $m$, $d$ and the choice of $\psi$.
\end{lemma}
\begin{proof}
  By considering the change of variable $x = \xi \varepsilon$, we
  observe that
  \begin{equation*}
    \begin{split}
   \||x|^m\depsilon(x)\|_{L^1(\Rd)} & = \int_{B_{\varepsilon r_0}}
   ||x|^m\depsilon(x)|\diff x \\
   & = \int_{B_{\varepsilon r_0}} \left||x|^m\frac{1}{\varepsilon^d}
     \psi\left(\frac{x}{\varepsilon}\right)\right|\diff x\\
   & = \int_{B_{r_0}} \left||\varepsilon\xi|^m\psi(\xi)\right|\diff \xi\\
   & \leq \| \psi \|_{L^\infty(\Rd)} ~~\|~|\xi|^m~\|_{L^1(B_{r_0})}~~\varepsilon^m.
 \end{split}
\end{equation*}
\end{proof}

We shall consider two common methods for choosing $\psi(x)$. In the
first case we choose a one dimensional function
$\psi_\rho\in L^\infty(\Real)$ so that $\psi_\rho$ is supported in
$[0,1]$, and we define the function $\psi$ as
\begin{equation}
  \label{e:delta-app-1}
  \psi(x) := I_d\psi_\rho(|x|), 
\end{equation}
where $I_d$ is a scaling factor, chosen so that $\psi(x)$ integrates
to one. With this choice, $\psi$ satisfies
Assumption~\ref{a:app-dirac} with $k$ at least equal to one (i.e., it
satisfies the first moments conditions).

The second construction method is usually referred to as tensor
product construction. We start from a $L^\infty(\Real)$ function
$\psi_{1d}$ that satisfies Assumption~\ref{a:app-dirac} for some $k$
in dimension one. Then we define $\psi(x)$ in dimension $d$ by
\begin{equation}\label{e:delta-app-2}
	\psi(x) := \prod_{i=1}^d \psi_{1d}(x_i), 
	\quad \text{ for } x = (x_1,\cdots, x_d)\in\Rd .
\end{equation}
The tensor product approximation $\psi(x)$ satisfies
Assumption~\ref{a:app-dirac} with {$r_0=\sqrt{d}$} and the same $k$ of $\psi_{1d}$. In
particular $k=0$ if $\psi_{1d}$ is not symmetric, and $k$ is equal to
at least one if $\psi_{1d}$ is symmetric.  We refer to \cite[Section
3]{MR3429589} for an in depth discussion on other possible choices of
Dirac approximation classes $\psi(x)$ with possibly higher order
moment conditions.

\subsection{Unbounded domains}
We begin by providing some results that follow from an application of
Young's inequality for convolutions:
\begin{lemma}[Young's inequality for convolutions~\cite{Young1912}]
  \label{l:young}
  Given $f,g\in L^2(\Rd)$ and $h\in L^1(\Rd)$,
  \begin{equation}
    \bigg|\int_\Rd\int_\Rd f(x)g(y)h(x-y)\diff x\diff y\bigg| \le
    \|f\|_{L^2(\Rd)}\|g\|_{L^2(\Rd)}\|h\|_{L^1(\Rd)},
    \label{e:youngs-inequality}
  \end{equation}
\end{lemma}

\begin{lemma}\label{l:reg-Rd}
  For $0\le s \le k+1$, let $v\in H^s(\Rd)$, and let $\vepsilon$ be
  defined by Definition~\ref{d:regularization}. Then there holds
\begin{equation}\label{i:reg-l2-esti-Rd}
	\|v-\vepsilon\|_{L^2(\Rd)} \preceq
        \varepsilon^{s}{\|v\|_{H^s(\Rd)}}, \qquad 0\le s \le k+1.
\end{equation}
\end{lemma}
\begin{proof}
  \boxed{1} We start by considering $v\in C^\infty_c(\Rd)$, and the
  case where $s$ is integer, and $1 \leq s \leq k+1$. By Taylor expansion, it is
  possible to expand $v(y)$ around an arbitrary point $x$ in a
  polynomial part $p_x$ and a residual part $r_x$, i.e.:
  \begin{equation*}
    \begin{split}
       v( y ) &= p_x(y) + r_x(y)\\
      & =\underbrace{\sum_{|\alpha|\leq s-1} \frac{D^\alpha v(x)}{\alpha!}
      (y-x)^\alpha}_{p_x \in P^{s-1}(\Rd)} + \underbrace{\sum_{|\beta|=s} R_\beta(y)(y-x)^\beta}_{r_x}, \\
      R_\beta( y ) & := \frac{|\beta|}{\beta!} \int_0^1
      (1-t)^{|\beta|-1}D^\beta v \big(x+t( y-x )\big) \diff t.
    \end{split}
  \end{equation*}
  Here $\alpha$ and $\beta$ are multi-indices. If we regularize $v$, since the regularization is a linear operator, we obtain:
  \begin{equation*}
    v^\varepsilon(y) = p_x^\varepsilon(y) + r_x^\varepsilon(y) =   p_x(y) + r_x^\varepsilon(y),
  \end{equation*}
 where in the last equality follows from Assumption~\ref{a:app-dirac}.2 so that $p_x^\varepsilon = p_x$ for any polynomial
  of order up to $k$.
  
  Given any $\theta\in L^2(\Rd)$, we deduce
  \[
    \begin{aligned}
      (v-\vepsilon, \theta)_{\Rd} & = \int_{\Rd}\left(v(x)-\int_{\Rd}
        \depsilon(x-y) v(y) \diff y\right)\theta(x) \diff x \\
      & = \int_{\Rd}\left(\underbrace{r_x(x)}_{=0}-\int_{\Rd}
        \depsilon(x-y) r_x(y) \diff y\right)\theta(x) \diff x,\\
      & = -\int_{\Rd}\int_\Rd \depsilon(x-y) r_x(y) \theta(x) \diff y.
      \diff x
    \end{aligned}
  \]
  Applying the definition of $r_x(y)$, Fubini's theorem, and by the change of variable
  $\xi =x+ t(y-x)$ for a fixed $x\in \Rd$, and $t\in (0,1)$, we have:
    \begin{equation*}
      \begin{split}
        (v-\vepsilon, \theta)_{\Rd}  =&  -\int_0^1 \int_{\Rd}\int_\Rd
    \sum_{|\beta|=s} (1-t)^{|\beta|-1} \frac{|\beta|}{\beta!}
        \\
        &D^\beta v \big(x+t( y-x )\big) (y-x)^\beta \depsilon(x-y)
        \theta(x) \diff y \diff x \diff t\\
        = &-\int_0^1 \int_{\Rd}\int_\Rd
        \sum_{|\beta|=s} (1-t)^{|\beta|-1} \frac{|\beta|}{\beta!}
        \\
        &D^\beta v (\xi) \left(\frac{\xi-x}{t}\right)^\beta \depsilon\left(-\frac{\xi-x}{t}\right)
        \theta(x) \frac{1}{t^d}\diff \xi \diff x \diff t.
      \end{split}
    \end{equation*}
    Applying Lemma~\ref{l:young} and \ref{l:L1-control} we get
    \begin{equation*}
      \begin{split}
        (v-\vepsilon, \theta)_{\Rd} &\leq
        |v|_{H^{s}(\Rd)}\|\theta\|_{L^2(\Rd)} \int_0^1
        \frac{(1-t)^{s-1}}{t^d (s-1)!}  \left\| \left|\frac{x}{t}\right|^{s}
          \depsilon\left(-\frac{x}{t}\right) \right\|_{L^1(\Rd)} \diff
        t \\
        & \preceq   \varepsilon^{s} |v|_{H^{s}(\Rd)}\|\theta\|_{L^2(\Rd)}.
      \end{split}
    \end{equation*}
    
   {\boxed{2} To show \eqref{i:reg-l2-esti-Rd} for $s=0$, we note that the regularization operator is $L^2$-stable, \ie for any $\theta\in L^2(\Rd)$,
   \[
  \begin{aligned}
    	(\vepsilon, \theta) &\le \int_\Rd\int_\Rd |v(y)\depsilon(x-y)\theta(x)|\diff y\diff x  \\
	&\le \|v\|_{L^2(\Rd)}\|\depsilon\|_{L^1(\Rd)}\|\theta\|_{L^2(\Rd)}\preceq  \|v\|_{L^2(\Rd)}\|\theta\|_{L^2(\Rd)} .
  \end{aligned}
  \]
  Here we applied again Young's inequality (Lemma~\ref{l:young}) together with Lemma~\ref{l:L1-control} for $m=0$. 
  Hence, $\|\vepsilon\|_{L^2(\Rd)}\preceq\|v\|_{L^2(\Rd)}$, and $\|v-\vepsilon\|_{L^2(\Rd)}\preceq\|v\|_{L^2(\Rd)}$ follows from the triangle 
  inequality.}

    \boxed{3} Taking the sup over $\theta$ with unit $L^2(\Rd)$ norm, and applying
    interpolation estimates between $s=0$ and $s=k+1$, the proof is
    complete by a density argument.
  
\end{proof}

The above lemma immediately implies the following convergence result
in $\Rd$:
  \begin{theorem}[Regularization estimates in
    $\Rd$]\label{t:estimates-Rd}
    Let $F\in H^m(\Rd)$,
    $m \in [-k-1,0]$. For $-k-1\leq s \leq m\leq 0$, there holds
    \begin{equation}\label{i:stability-Rd-F}
      \|F-\Fepsilon\|_{H^{s}(\Rd)} \preceq \varepsilon^{m-s}\|F\|_{H^m(\Rd)} .
    \end{equation}
    Moreover, let $v\in H^m(\Rd)$, $m \in [0,k+1]$. For $s\in [-k-1,m]$ so that
    $m-s \leq k+1$, there holds
    \begin{equation}\label{i:stability-Rd}
      \|v-\vepsilon\|_{H^{s}(\Rd)} \preceq \varepsilon^{m-s}\|v\|_{H^m(\Rd)} .
    \end{equation}
    Here we identify $v$ in the negative Sobolev space by the duality pairing:
    $\langle v,\cdot\rangle = (v,\cdot)_{L^2(\Rd)}$.
  \end{theorem}

\begin{proof}
Let us show the desired estimates in three steps.

\boxed{1} For $v\in C^\infty_c(\Rd)$, the definition of the weak derivative of $D^\alpha \depsilon$ for $|\alpha|\le k+1$ yields that for
  $x\in \Rd$,
\[
\begin{aligned}
	D^\alpha \vepsilon(x) &= \int_\Rd v(y)D^\alpha\depsilon(x-y)\diff y \\
	& = \int_\Rd D^\alpha v(y) \depsilon(x-y)\diff y = (D^\alpha v)^\varepsilon(x) .
\end{aligned}
\]
Hence, we apply Lemma~\ref{l:reg-Rd} to $(\sum_{|\alpha|=k+1} D^\alpha  v)^\varepsilon$ with $s=0$ to get
\begin{equation}\label{i:hs-stability-Rd}
	\|v-\vepsilon\|_{H^{k+1}(\Rd)} \preceq \|v\|_{H^{k+1}(\Rd)}.
\end{equation}
Interpolating the estimates between \eqref{i:reg-l2-esti-Rd}
and \eqref{i:hs-stability-Rd} implies \eqref{i:stability-Rd} for
$0\leq s \leq m \leq k+1$.

\boxed{2} For {$F\in H^{m}(\Rd)$} with $-k-1\le m\le 0$, we have
\begin{equation}
  \label{i:estimate-Hminusk-minus1}
  \begin{split}
    \| F-F^\varepsilon\|_{H^{m}(\Rd)}  & :=\sup_{w\in H^{-m}(\Rd)}
    \frac{\langle F- F^\varepsilon, w \rangle}{\|w\|_{H^{-m}(\Rd)}}  \\
    & :=   \sup_{w\in H^{-m}(\Rd)}
    \frac{\langle F,  w-w^\varepsilon\rangle}{\|w\|_{H^{-m}(\Rd)}} \\
    & \preceq \sup_{w\in H^{-s}(\Rd)} \frac{\| F \|_{H^{m}(\Rd)} \|
      w-w^\varepsilon\|_{H^{-m}(\Rd)}}{\|w\|_{H^{-m}(\Rd)}}\\
    & \preceq \| F \|_{H^{m}(\Rd)}.
\end{split}
\end{equation}
Similarly, for $F\in H^{m}(\Rd)$ with $-k-1\le m\le 0$,
\begin{equation}
  \label{i:estimate-Hminus-s}
  \begin{split}
    \| F-F^\varepsilon\|_{H^{-k-1}(\Rd)}  & \preceq \sup_{w\in H^{k+1}(\Rd)} \frac{\| F \|_{H^{m}(\Rd)} \|
      w-w^\varepsilon\|_{H^{-m}(\Rd)}}{\|w\|_{H^{k+1}(\Rd)}}\\
    & \preceq \varepsilon^{k+1+m} \| F \|_{H^{m}(\Rd)} .
\end{split}
\end{equation}
So the first assertion follows from the interpolation {between
\eqref{i:estimate-Hminusk-minus1} and
\eqref{i:estimate-Hminus-s}.}

\boxed{3} For $v\in H^m(\Rd)$, $m \in [0,k+1]$, interpolating the result ($s\leq0$) between
$\| v-v^\varepsilon \|_{H^{s}(\Rd)} \preceq
\varepsilon^{-s}\|v\|_{L^2(\Rd)}$ and
$\| v-v^\varepsilon \|_{L^{2}(\Rd)} \preceq
\varepsilon^{m}\|v\|_{H^m(\Rd)}$ concludes the proof of the second desired estimate, with $m-s\leq k+1$.
\end{proof}

\subsection{Bounded domains}

The generalization of the previous results in bounded domains is
non-trivial, due to the presence of boundaries. We start by providing
some results that work well when restricting $v$ to a region which is
strictly contained in $\Omega$. Let this region be defined by a
Lipschitz domain $\omega\subset\Omega$, and assume that there exists a
positive constant $c_0$ such that
 \begin{equation}\label{i:interface-assumption}
   \dist(\partial\omega,\partial\Omega)> c_0.
 \end{equation}

The following result relies on the interior regularity of
$v$. More precisely, according to \eqref{i:interface-assumption}, we
assume that the regularization parameter satisfies
$\varepsilon\le \varepsilon_0\leq c_0$
for some fixed $\varepsilon_0$, and we set
\begin{equation}\label{e:extension}
	\Oepsilon = \bigcup_{x\in\omega}\Bepsilon(x) \subset \Omega.
\end{equation}
Assuming that $v\in H^s(\Oe)\cap L^1(\Omega)$ with $s\in [0,k+1]$, we 
next provide similar error estimates compared with the unbounded case
in the previous section. We note that results below are instrumental 
to our error estimates for the numerical approximation of the model
problems in the next section.

\begin{lemma}[$L^2(\omega)$ estimate]\label{l:k-reg-omega}
  For $0\le s \le k+1$, and
  $\varepsilon \leq \varepsilon_0\leq c_0$, let
  {$v\in L^1(\Omega)\cap H^s(\omega^{\varepsilon_0})$}, and let
  $\vepsilon$ be defined as in Definition~\ref{d:regularization}. Then
 there holds
\begin{equation}\label{i:k-reg-l2-esti}
	\|v-\vepsilon\|_{L^2(\omega)} \preceq \varepsilon^{s} {\|v\|_{H^s(\Oe)}}.
\end{equation}
\end{lemma}
\begin{proof}
  We denote by $\widetilde\cdot$ the zero extension from $\omega$, $\Oe$ or $\Omega$ to $\Rd$. 
  Since $\varepsilon<\dist(\partial\omega,\partial\Omega)$, for any $x\in \omega$, $\Bepsilon(x)\subset\Omega$ and
  \[
    \int_{\Omega} \depsilon(x-y) \diff y = \int_{\Bepsilon(x)} \depsilon(x-y) \diff y = \int_\Rd \depsilon(y)\diff y = 1.
  \]
  For any $\theta\in L^2(\omega)$ and $v\in C_0^\infty(\overline\Oe)$, we follow the proof of Lemma~\ref{l:reg-Rd} to get
  \begin{equation}\label{i:reg-diff}
    \begin{aligned}
      (v-\vepsilon, \theta)_\omega &= \int_\omega v(x) \theta(x) \bigg(\int_{\Omega} \depsilon(x-y) \diff y\bigg) \diff x \\
      &\qquad\qquad - \int_\omega \int_\Omega v(y)\depsilon(x-y) \theta(x) \diff y\diff x\\
      &= -\int_\omega \int_\Oe (v(y) - v(x))\depsilon(x-y) \theta(x) \diff y\diff x \\
      &= -\int_\omega \int_\Oe r_x(y)\depsilon(x-y) \theta(x) \diff y\diff x \\
      &\preceq \int_0^1 \int_{\Rd}\int_\Rd
      |\widetilde{D^\beta v} (\xi)| \bigg|\left(\frac{\xi-x}{t}\right)^\beta\depsilon\left(-\frac{\xi-x}{t}\right) \bigg|
      |\widetilde{\theta}(x)| \frac{1}{t^d}\diff \xi \diff x \diff t.
    \end{aligned}
  \end{equation}
  Here we apply again Assumption~\ref{a:app-dirac}.2 for the last
  equality above. When it comes to the last inequality in \eqref{i:reg-diff},
  we note that for a fixed $x\in\omega$ and for any $y\in\Oe$, the change of
  variable $\xi = t(y-x)+x $ belongs to $\Oe$ for any $t\in (0,1)$. 
  {Hence we
  proceed following the proof of Lemma~\ref{l:reg-Rd}, Step 1 and apply
  Lemma~\ref{l:young} again to conclude the proof for a positive integer $s$. 
  Replacing $v$ and $\theta$ with $\widetilde v$ and $\widetilde \theta$ in the proof of  Lemma~\ref{l:reg-Rd}, Step 2, we obtain \eqref{i:k-reg-l2-esti}
  with $s=0$. The assertion for any $s\in[0,k+1]$ follows from the interpolation between $s=0$ and $s=k+1$.}
\end{proof}

\begin{corollary}[$H^s(\omega)$ estimate]\label{c:k-reg-omega}
  For $0\le s\leq m \le k+1$, and
  $\varepsilon_0\leq c_0$, let
  $v\in L^1(\Omega)\cap H^m(\Oe)$, and let
  $\vepsilon$ be defined by Definition~\ref{d:regularization}. Then there
  exists $\varepsilon_1 > 0$ small enough so that for $\varepsilon < \varepsilon_1$,
\begin{equation}\label{i:k-reg-Hm-esti}
	\|v-\vepsilon\|_{H^s(\omega)} \preceq \varepsilon^{m-s} \|v\|_{H^m(\Oe)}.
\end{equation}
\end{corollary}
\begin{proof}
Let $m$ be a positive integer. Integration by parts yields that for $v\in
  C^\infty_c(\overline\Oe)$ with $x\in \omega^{\varepsilon_{0}/2}$ and for $\varepsilon < \varepsilon_0/2$,
\[
\begin{aligned}
    D v^\varepsilon(x) &= \int_\Omega D_x \delta^\varepsilon(x-y) v(y) \diff y
    = \int_\Omega -D_y \delta^\varepsilon(x-y) v(y) \diff y \\ 
    &= \int_\Omega \delta^\varepsilon(x-y) D_y v(y) \diff y = (D v)^\varepsilon(x).
\end{aligned}
\]
Repeating the above argument shows that given a multi-index $\beta$ so that $|\beta|=m$,
$D^\beta \vepsilon(x) = {(D^\beta v)^\varepsilon(x)}$ for $x\in \omega^{\varepsilon_0/2^m}$
when $\varepsilon < \varepsilon_0/2^m$.
  We then apply \eqref{i:k-reg-l2-esti} in Lemma~\ref{l:k-reg-omega} with $s=0$ to
  $(\sum_{|\beta|\leq m} D^\beta v)^\varepsilon$ and a density argument to obtain that
  \[
  	\|v-\vepsilon\|_{H^m(\omega)} \preceq \|v\|_{H^m(\omega^{\varepsilon_0/2^m})} \leq \|v\|_{H^m(\omega^{\varepsilon_0})} .
  \]
  Interpolating the above estimate with \eqref{i:k-reg-l2-esti} yields the desired estimate.
\end{proof}

 Given a functional $F\in H^s(\Omega)$ with $-k-1\le s\le 0$, we say 
 $\text{supp}_m(F)\subseteq \omega$ for some $m\in [s,0]$  if 
 $F \in H^{m}(\omega)$ for some $s\in [-k-1,m]$ so that
  \[
  	\langle F , w\rangle_{H^s(\Omega), H^{-s}(\Omega)} 
	\preceq \|F\|_{H^{m}(\omega)} \|w\|_{H^{-m}(\omega)}, \quad\text{for all } w\in H^{-s}(\Omega) .
  \]

\begin{lemma}\label{l:estimate-Hminus-s-Omega}
  Let $F\in H^{s}(\Omega)$ with $-k-1\leq s\leq 0$, and $\emph{supp}_m(F)\subseteq \omega$ for some $m\in [s,0]$.  Then, there holds
  \begin{equation}
    \label{e:estimates-H-minus-s-Omega}
    \|F-\Fepsilon\|_{H^{s}(\Omega)} \preceq \varepsilon^{m-s} \|F\|_{H^{m}(\omega)} .	
  \end{equation}
\end{lemma}
\begin{proof}
  The proof is identical to the proof of Theorem~\ref{t:estimates-Rd},
  replacing the application of Lemma~\ref{l:reg-Rd} with that of
  Lemma~\ref{l:k-reg-omega} and Corollary~\ref{c:k-reg-omega}.
\end{proof}

The above results are summarized in the following theorem:
  \begin{theorem}[Regularization estimates in
    $\Omega$]\label{t:estimates-omega}
    Let $\omega$ be such that {\eqref{i:interface-assumption}} holds,
    i.e.
    \begin{equation*}
       \dist(\partial\omega,\partial\Omega)> c_0.
    \end{equation*}
    Define the extension of $\omega$ as in {\eqref{e:extension}},
    i.e.,
    \begin{equation*}
      \Oepsilon = \bigcup_{x\in\omega}\Bepsilon(x) \subseteq \Omega,
    \end{equation*}
    and let $\varepsilon \leq \varepsilon_0\leq c_0$.
    For $0\le s \le m \leq k+1$ where $k$ is the order of the moments
    conditions satisfied by $\depsilon$ as in
    Assumption~\ref{a:app-dirac}.2, we have:
    \begin{itemize}
    \item If $v\in H^m(\omega^{\varepsilon_0}) \cap L^1(\Omega)$, then
      \begin{equation*}
        \|v-v^\varepsilon\|_{H^s(\omega)} \preceq  \varepsilon^{m-s}
        \|v\|_{H^m(\omega^{\varepsilon_0})}.
      \end{equation*}
    \item If $F \in H^{-m}(\Omega)$ and $\emph{supp}_{-s}(F) \subseteq
      \omega$, then
      \begin{equation*}
         \|F-\Fepsilon\|_{H^{-m}(\Omega)} \preceq \varepsilon^{m-s} \|F\|_{H^{-s}(\omega)}.
      \end{equation*}
    \end{itemize}
  \end{theorem}

\section{Model problem}\label{ss:model-problem}
We are now in a position to apply the results of
Theorem~\ref{t:estimates-omega} to a model
elliptic problem.  Let $A(x)$ be a symmetric $d\times d$ matrix. We
assume that all entries of $A(x)$ are in $C^1(\overline\Omega)$,
uniformly bounded, and that $A(x)$ is positive definite, \ie there
exist positive constants $a_0,a_1$ satisfying
\[
	a_0|\xi|^2\le \xi^{\Tr} A(x) \xi \le a_1|\xi|^2, \quad \forall \xi\in\Rd \text{ and } x\in \overline\Omega .
\]
We also set $c(x)$ to be a nonnegative function in $C^{0,1}(\overline\Omega)$. 
We now define the forcing data for our model problem.
We set $F\in H^{-1}(\Omega)$. We assume that for some $s\in(0,\tfrac12]$, $\text{supp}_{s-1}(F)\subset\omega$.

Based on the above definitions, our model problem reads: find the distribution $u$ satisfying
\begin{equation}\label{e:distributional}
	\begin{aligned}
		-\DIV(A(x)\GRAD u) + c(x)u = F, \quad &\text{ in } \Omega, \\
		u = 0, \quad &\text{ on } \partial\Omega .
	\end{aligned}
\end{equation}
To approximate Problem \eqref{e:distributional} using the finite element method, we shall consider its variational formulation: find $u\in V:=\Hunz$ satisfying
\begin{equation}\label{e:variational}
	\calA(u,v) = \langle F, v \rangle_{V', V}, \quad \forall v\in V,
\end{equation}
where
\[
	\calA(v,w) = \int_{\Omega} \GRAD v^{\Tr} A(x) \GRAD w + c(x) vw\diff x, \quad\text{ for } v,w\in V. 
\]

\subsubsection*{Regularity}
Our error estimates rely on standard regularity results for elliptic problems: given $g\in V'$, let $T :  V' \to V$ be the solution operator satisfying
\begin{equation}\label{e:general-variational}
	\calA(Tg,v)  = \langle g,v\rangle_{V',V}, \qquad \forall v\in V .
\end{equation}
We first note that if $g\in L^2(\Omega)$, we identify $\langle g,\cdot\rangle_{V',V}$ with $(g,\cdot)_\Omega$ and hence $Tg$ has the $H^2$ interior regularity, \ie given a subset $K$ such that $\overline K\subset\Omega$, $Tg\in H^2(K)$ and
\begin{equation}\label{i:interior-regularity}
	\|Tg\|_{H^2(K)} \preceq \|g\|_{L^2(\Omega)},
\end{equation}
where the hiding constant depends on $K$ and $\Omega$; we refer to \cite[Theorem 1 of Section 6.1]{MR2597943} and \cite[Theorem 5.33]{MR2895178} for a standard  proof. The following assumption provides the regularity of $Tg$ up to the boundary 
\begin{assumption}[elliptic regularity]
\label{a:elliptic-regularity}
There exists $r\in(0,1]$ and a positive constant $C_r$ satisfying
\[
	\|Tg\|_{H^{1+r}(\Omega)}\leq C_r \|g\|_{H^{-1+r}(\Omega)} .
\]
\end{assumption}
As an example, consider the case where $\Omega$ is a polytope, $A(x)$ is the identity matrix, and $c(x)=0$, \ie $\calA$ becomes the Dirichlet form 
\begin{equation}\label{e:dirichlet}
	\calA(v,w)=\int_\Omega \GRAD v^{\Tr}\GRAD w\diff x, \quad\forall v,w\in V.
\end{equation}
Based the regularity results provided by \cite{MR961439}, $r$ in Assumption~\ref{a:elliptic-regularity} is between $\tfrac12$ and 1 and can be decided by the shape of $\Omega$. Assumption~\ref{a:elliptic-regularity} also implies that the solution $u$ in \eqref{e:variational} belongs to $H^{1+\min\{s,r\}}(\Omega)\cap\Hunz$.
 
\subsection{Analysis of a regularized problem}
Now we are ready to define a regularized problem of \eqref{e:variational}: find $\uepsilon\in V$ satisfying
\begin{equation}\label{e:regularized}
	\calA(\uepsilon, v) = \langle \Fepsilon, v\rangle_{V', V}, \qquad \forall v\in V .
\end{equation}
\begin{remark}
  Notice that we denote with $\uepsilon$ the solution to
  Problem~\ref{e:regularized}, and in this case the superscript
  $\varepsilon$ does not denote a regularization (hence the different
  font used for $\uepsilon$).
\end{remark}
In order to bound the error between $u$ and $\uepsilon$, we use the boundedness and coercivity of $\calA(\cdot,\cdot)$, and obtain that
\[
\begin{aligned}
	\|u-\uepsilon\|_{H^1(\Omega)}^2 &\preceq \calA(u-\uepsilon,u-\uepsilon) \\
	&\preceq \langle F-\Fepsilon,u-\uepsilon\rangle_{V',V} \\
	&\le \|F-\Fepsilon\|_{H^{-1}(\Omega)}\|u-\uepsilon\|_{H^1(\Omega)} .
\end{aligned}
\]
This implies that
\[
	\|u-\uepsilon\|_{H^1(\Omega)}\preceq \|F-\Fepsilon\|_{H^{-1}(\Omega)} .
\]
Applying Lemma~\ref{l:estimate-Hminus-s-Omega} yields
\begin{theorem}[$H^1(\Omega)$ error estimate]\label{t:h1-reg}
Under the assumptions in Proposition~\ref{l:estimate-Hminus-s-Omega}, let $u$ and $\uepsilon$ be the solutions of problem \eqref{e:variational} and \eqref{e:regularized}, respectively. Then, there holds
\[
	\|u-\uepsilon\|_{H^1(\Omega)} \preceq \varepsilon^s \|F\|_{H^{s-1}(\omega)} .
\]
\end{theorem}
\subsection{$L^2(\Omega)$ error estimate}
We next show the convergence rate for $\uepsilon$ in $L^2(\Omega)$
norm. To this end, we additionally assume that $\depsilon$ satisfies
the moments conditions in Assumption~\ref{a:app-dirac}.2 with $k\geq1$.

\begin{theorem}[$L^2(\Omega)$ error estimate]\label{t:reg-estimate}
  Under the assumptions of Lemma~\ref{l:estimate-Hminus-s-Omega}, and
  Assumption~\ref{a:app-dirac}.2 with $k\geq1$, let $u$ and
  $\uepsilon$ be the solutions of problems \eqref{e:variational} and
  \eqref{e:regularized}, respectively. For a fixed $\varepsilon_0<c_0$ and 
  $\varepsilon< \varepsilon_0/2$, there holds
\[
	\|u-\uepsilon\|_{L^2(\Omega)} \preceq \varepsilon^{s+1}\|F\|_{H^{s-1}(\omega)} .
\]
\end{theorem}
\begin{proof}
We consider the following dual problem: find $z\in V$ such that
\[
	\calA(v, z) = (u-\uepsilon, v)_\Omega, \qquad\forall v\in V .
\]
Hence, we choose $v=u-\uepsilon$ and obtain that
\begin{equation}\label{e:reg-estimate-1}
\begin{aligned}
	\|u-\uepsilon\|_{L^2(\Omega)}^2 &=\calA(z, u-\uepsilon) \\
	&= \langle F-\Fepsilon, z\rangle_{V',V} =  \langle F, z-\zepsilon \rangle_{H^{-1}(\Omega),H^1(\Omega)}.
\end{aligned}
\end{equation}
Due to the interior regularity of $z$, $u-\uepsilon\in \Hunz\subset L^2(\Omega)$ implies that 
\[
	{\|z\|_{H^{2}(\omega^{\varepsilon_0})}}\preceq \|u-\uepsilon\|_{L^2(\Omega)} .
\] 
We continue to estimate the right hand side of \eqref{e:reg-estimate-1} by
\begin{equation}\label{i:reg-estimate-2}
\begin{aligned}
	  \langle F, z-\zepsilon \rangle_{H^{-1}(\Omega),H^1(\Omega)} 
	 & \preceq \|F\|_{H^{s-1}(\omega)} \|z-\zepsilon\|_{H^{1-s}(\omega)} \\
	& \preceq \varepsilon^{s+1}\|F\|_{H^{s-1}(\omega)} \|z\|_{H^{2}(\omega^{\varepsilon_0})} \\
	& \preceq  \varepsilon^{s+1}\|F\|_{H^{s-1}(\omega)} 
	\|u-\uepsilon\|_{L^2(\Omega)},
\end{aligned}
\end{equation}
where in the second inequality, we invoke Lemma~\ref{l:k-reg-omega} for $z$. Combing \eqref{e:reg-estimate-1} and \eqref{i:reg-estimate-2} concludes the proof of the theorem.
\end{proof}

\begin{remark}\label{r:bc}
  We have to point out that the estimates in Theorem~\ref{t:h1-reg}
  and \ref{t:reg-estimate} also hold for Problem
  \eqref{e:distributional} with other boundary conditions. These error
  estimates depend on the smoothness of the test function and of the
  solution for the dual problem in the neighborhood of $\omega$ while
  $\omega$ is away from the boundary. The analysis of the case where
  $\omega$ is attached to the boundary $\partial\Omega$ is left aside
  for future investigation.
\end{remark}

\section{Application to immersed methods}\label{s:ib}
The general results presented in the previous section can be applied
immediately to immersed interface and {immersed} boundary
methods~\cite{Heltai-2008-a, HeltaiRotundo-2019-a, 
HeltaiCaiazzo-2018-a, AlzettaHeltai-2020-a}. In this
section, we consider an interface problem whose variational
formulation can be written as in the model problem
\eqref{e:variational}, and we shall consider its finite element
approximation using the regularization of the forcing data given by
the application of Definition~\ref{d:regularization} to functions in
negative Sobolev spaces. 
\subsection{An interface problem via immersed methods}
Let $\Gamma\subset\Omega$ be a closed Lipschitz interface with
co-dimension one. We set $\omega$ be the domain inside $\Gamma$, and
we assume that $\omega$ satisfies the assumptions introduced in
Section~\ref{s:prelim} (see Figure~\ref{fig:domain}).
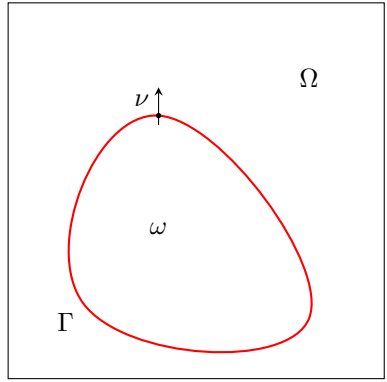
\begin{figure}[h]
  \tikzsetnextfilename{domain}
  \centering
  \begin{tikzpicture}
    
    \draw (0,0) rectangle (5,5) node[anchor=south west](O) {};
    
    \node at (1,1) (G){};
    \node at (2,3.5) (N){};
    \node at (2,4) (Np){};
    
    \draw [red, thick] plot [smooth cycle, tension=1]
    coordinates { (G)  (4,.8)  (N)};
    
    \node at (4,4) {$\Omega$};
    \node at (2,2) {$\omega$};
    \node [below left] at (G) {$\Gamma$};
    \node [above left] at (N) {$\nu$};
    
    \draw [-stealth] (N.south) -- (Np);
    \fill (N) circle (1pt);
    
  \end{tikzpicture}
  \caption{Domain representation}
  \label{fig:domain}
\end{figure}

Given $f\in H^{s-1/2}(\Gamma)$ with $s\in[0,\tfrac12]$, we consider the following Poisson problem
\begin{equation}\label{e:strong}
	\begin{aligned}
		-\Delta u = 0, \quad &\text{ in } \Omega\backslash\Gamma, \\
		\llbracket u \rrbracket = 0, \quad &\text{ on } \Gamma,\\
		\bigg\llbracket \frac{\partial u}{\partial \nu} \bigg\rrbracket = f, \quad &\text{ on } \Gamma,\\
		u = 0, \quad &\text{ on } \partial\Omega.
	\end{aligned}
\end{equation}
For $x\in\Gamma$, $\nu(x)$ denotes the normal vector and $\tfrac{\partial u}{\partial \nu}$ denotes the corresponding normal derivative. We also use the notation $\llbracket\cdot\rrbracket$ for the jump across $\Gamma$. More precisely speaking, letting $u^-=u\big|_{\omega}$ and $u^+=u\big|_{\Omega\backslash\overline\omega}$, the jump across $\Gamma$ is defined by
\[
	\llbracket u \rrbracket = u^+\big|_{\Gamma} - u^-\big|_{\Gamma} .
\]
The above problem can be reformulated in the entire domain $\Omega$ using a singular forcing term that induces the correct jump on the gradient of the solution. (cf. \cite{MR2441884}). The variational fomulation of the new problem is provided by \eqref{e:variational} where $\calA$ is the Dirichlet form \eqref{e:dirichlet} and 
\begin{equation}\label{e:dirac}
  F = \cM f := \int_{\Gamma} \delta(x-y) f(y) \diff\sigma_y .
\end{equation}
Here $\delta$ denotes the $d$-dimensional Dirac delta
distribution, and Eq.~\eqref{e:dirac} should be interpreted variationally,
i.e.,
\begin{equation}\label{e:variational-forcing}
    \langle F, v \rangle_{V', V} := \langle  \cM f, v\rangle_{V', V} 
    := \int_{\Gamma} fv \diff\sigma, \quad \forall v\in V.
\end{equation}
We note that for any $s$ in $[0,\tfrac12)$, $\cM$ is a bounded map from $H^{s-1/2}(\Gamma)$ to $H^{s-1}(\Oee)\cap H^{-1}(\Omega)$ and hence the variational formulation makes sense. 
In fact, for $v\in H^1(\Omega)$,
\[
\begin{aligned}
	\bigg| \int_\Gamma f(y)v(y)\diff\sigma_y \bigg| 
	&\le \|f\|_{H^{s-1/2}(\Gamma)}\|v\|_{H^{1/2-s}(\Gamma)} \\
	&\le \|f\|_{H^{s-1/2}(\Gamma)}\|v\|_{H^{1-s}(\omega)} \le \|f\|_{H^{s-1/2}(\Gamma)}\|v\|_{H^{1}(\Omega)} .
\end{aligned}
\]
Here we apply the trace Theorem (e.g. \cite[Theorem 1.5.1.2]{MR1742312}) for the second inequality. The above estimate together with \eqref{e:variational-forcing} shows that $F\in V'$. It also implies that $F\in H^{s-1}(\Oee)$ and hence $\text{supp}(F)_{s-1}\subseteq \omega$. So $F$ satisfies the setting in Section~\ref{ss:model-problem} and we can apply the results from Theorem~\ref{t:h1-reg} and \ref{t:reg-estimate} to the interface problem to get
\begin{equation}\label{i:error-h1-l2}
	\|u-\uepsilon\|_{L^2(\Omega)} + \varepsilon \|u-\uepsilon\|_{H^1(\Omega)} \preceq \varepsilon^{s+1}\|f\|_{H^{s-1/2}(\Gamma)} .
\end{equation}

We note that Fubini's Theorem yields that
\[
\begin{aligned}
	\langle F, \vepsilon \rangle_{H^{-1}(\Omega), H^1(\Omega)} &=\int_\Gamma f(x) \int_\Omega v(y) \depsilon(x-y)\diff x \diff\sigma_y \\
	& = \int_\Omega v(y) \int_\Gamma f(x)  \depsilon(x-y) \diff
        \sigma_y \diff x 
\end{aligned}
\]
and hence
\begin{equation}
  \label{e:classical-Fepsilon}
  \begin{split}
    \Fepsilon(x) &=  \int_\Gamma f(y)  \depsilon(y-x) \diff y\\
    \Big(&= \int_\Gamma f(y)  \depsilon(x-y) \diff y \qquad \text{ if } k
    \geq 1\Big)
  \end{split}
\end{equation}
which is the classical formulation of $\Fepsilon$ that can be found in
the literature of the Immersed Boundary Method~\cite{Peskin-2002-a},
where $\depsilon$ is always taken to be even (i.e., $k\geq1$), and
$\Fepsilon$ is introduced as the regularization of $f$ on $\Gamma$,
via the $d$-dimensional Dirac approximation $\depsilon$.

\subsection{Finite Element Approximation of the regularized problem}\label{s:fem}
In this section, we consider a finite element approximation of the
regularized problem \eqref{e:regularized}. Assume that $\Omega$
is polytope and let $\{\mathcal T_h(\Omega)\}_{h>0}$ be a family of conforming
subdivisions of $\overline\Omega$ made of simplices with $h$ denoting
their maximal size. We assume that $\mathcal T_h(\Omega)$ are
shape-regular and quasi-uniform in the sense of
\cite{MR2050138,MR1930132}.

Let $\mathbb V_h$ be the space of continuous piecewise linear functions subordinate to $\mathcal T_h(\Omega)$ that vanish on $\partial\Omega$. Let $I_h : H^1_0(\Omega)\to \mathbb V_h$ be the Scott-Zhang interpolation \cite{MR1011446} which has the following approximation property
\begin{equation}\label{i:scott-zhang}
	\|v-I_h v\|_{H^1(\Omega)}\preceq h^s\|v\|_{H^{1+s}(\Omega)},\quad \text{for } v\in H^{1+s}(\Omega)\cap\Hunz .
\end{equation}

The discrete counterpart of the regularized problem \eqref{e:regularized} reads: find $\ueh\in \mathbb V_h$ satisfying
\begin{equation}\label{e:discrete}
	\calA(\ueh,v_h) = \langle \Fepsilon,v_h\rangle_{V',V},\qquad\forall v_h\in \mathbb V_h .
\end{equation}
For $v_h\in \mathbb V_h$, $\langle\Fepsilon,v_h\rangle_{V',V}$ can be
computed by using a quadrature formula on $\Gamma$ and a quadrature
formula on $\tau\in\mathcal T_h(\Omega)$. We refer to the next Section
for the details of the implementation.

The following theorem shows the error between $u$ and its final approximation $\ueh$.
\begin{theorem}[$H^1(\Omega)$ error estimate]\label{t:discrete-h1-estimate}
Let $u$ and $\ueh$ be the solutions to \eqref{e:variational} and \eqref{e:discrete}, respectively. Under Assumption~\ref{a:elliptic-regularity} and \ref{a:app-dirac}, we have
\[
	\|u-\ueh\|_{H^1(\Omega)}\preceq (h^{\min\{s,r\}}+\varepsilon^{s})\|f\|_{H^{s-1/2}(\Gamma)} .
\]
\end{theorem}
\begin{proof}
The coercivity of $\calA(\cdot,\cdot)$ implies that
$\calA(\cdot,\cdot)$ is also $\mathbb V_h$ elliptic. The first Strang's Lemma (see, e.g. \cite[Theorem 4.1.1]{MR1930132}) yields
\[
	\|u-\ueh\|_{H^1(\Omega)}\preceq \inf_{v_h\in \mathbb V_h}\|u-v_h\|_{H^1(\Omega)} 
	+\sup_{w_h\in \mathbb V_h}\frac{\langle F-\Fepsilon,w_h\rangle_{H^{-1}(\Omega),H^1(\Omega)}}{\|w_h\|_{H^1(\Omega)}}
\]
Setting $v_h=I_h u$ and invoking \eqref{i:scott-zhang} together with Assumption~\ref{a:elliptic-regularity} and Lemma~\ref{l:estimate-Hminus-s-Omega}, we conclude that
\[
\begin{aligned}
	\|u-\ueh\|_{H^1(\Omega)}&\preceq \|u-I_h u\|_{H^1(\Omega)} + \varepsilon^{s}\|f\|_{H^{s-1/2}(\Gamma)} \\
	&\preceq h^{\min\{s,r\}}\|u\|_{H^{1+\min\{s,r\}}(\Omega)}+\varepsilon^{s}\|f\|_{H^{s-1/2}(\Gamma)} \\
	& \preceq (h^{\min\{s,r\}}+\varepsilon^{s})\|f\|_{H^{s-1/2}(\Gamma)} .
\end{aligned}
\]
\end{proof}
We next show a $L^2(\Omega)$ error estimate between $u$ and $\ueh$.
\begin{theorem}[$L^2(\Omega)$ error estimate]\label{t:discrete-l2-estimate}
Following the settings from Theorem~\ref{t:discrete-h1-estimate}, we
additionally assume that $\depsilon$ satisfies
Assumption~\ref{a:app-dirac}.2 with $k\geq1$. Then we have
\[
	\|u-\ueh\|_{L^2(\Omega)}\preceq (h^{r+\min\{s,r\}}+h^r\varepsilon^s+\varepsilon^{1+s})\|f\|_{H^{s-1/2}(\Gamma)} .
\]
\begin{proof}
We first provide a bound on the error between $\uepsilon$ and $\ueh$ under the regularity assumption of $f$. In fact, the triangle inequality together with Theorem~\ref{t:reg-estimate} and \ref{t:discrete-h1-estimate} implies that
\[
\begin{aligned}
	\|\uepsilon-\ueh\|_{H^1(\Omega)} &\le \|u-\uepsilon\|_{H^1(\Omega)} + \|u-\ueh\|_{H^1(\Omega)} \\
	&\preceq (h^{\min\{s,r\}}+\varepsilon^{s})\|f\|_{H^{s-1/2}(\Gamma)} .
\end{aligned}
\]
Next we bound $\|\uepsilon-\ueh\|_{L^2(\Omega)}$ using a duality argument. Let $\mathtt Z\in \Hunz$ satisfy that
\[
	\calA(\mathtt Z, v) = (\uepsilon-\ueh, v)_\Omega, \qquad\forall v\in \Hunz .
\]
Hence $\mathtt Z\in H^{1+r}(\Omega)$. Applying Galerkin orthogonality 
\[
	\calA(\uepsilon-\ueh, v_h) = 0,\qquad\forall v_h\in\mathbb V_h,
\]
we obtain that
\[
\begin{aligned}
	\|\uepsilon-&\ueh\|_{L^2(\Omega)}^2 = \calA(\mathtt Z, \uepsilon-\ueh) \\
	& = \calA(\mathtt Z - I_h \mathtt Z, \uepsilon-\ueh) \le \|\mathtt Z - I_h\mathtt Z\|_{H^1(\Omega)} \|\uepsilon-\ueh\|_{H^1(\Omega)} \\
	& \preceq h^r \|\mathtt Z \|_{H^{1+r}(\Omega)}(h^{\min\{s,r\}}+\varepsilon^{s})\|f\|_{H^{s-1/2}(\Gamma)} .
\end{aligned}
\]
This, together with the regularity estimate $\|\mathtt Z\|_{H^{1+r}(\Omega)}\preceq \|\uepsilon-\ueh\|_{L^2(\Omega)}$, shows that
\[
	\|\uepsilon-\ueh\|_{L^2(\Omega)} \preceq (h^{r+\min\{s,r\}}+h^r\varepsilon^{s})\|f\|_{H^{s-1/2}(\Gamma)} .
\]
The triangle inequality $\|u-\ueh\|_{L^2(\Omega)}\le \|u-\uepsilon\|_{L^2(\Omega)}+\|\uepsilon-\ueh\|_{L^2(\Omega)}$ together with the above estimate and the $L^2(\Omega)$ estimate in Theorem~\ref{t:reg-estimate} concludes the proof of the theorem.
\end{proof}
\end{theorem}
We end this section with some remarks to further explain the error estimates and also for the numerical experiments in the next section.
\begin{remark}[knowing the regularity of the solution]
If the regularity of the solution is known, e.g., $u\in H^{1+\beta}(\Omega)$ for some $\beta\in(0,1]$, we can apply the interpolation estimate \eqref{i:scott-zhang} with $s=\beta$ in Theorem~\ref{t:discrete-h1-estimate} and \ref{t:discrete-l2-estimate} to get
\[
	\|u-\ueh\|_{H^1(\Omega)}\preceq (h^\beta+\varepsilon^{s})\|f\|_{H^{s-1/2}(\Gamma)} 
\]
and
\[
	\|u-\ueh\|_{L^2(\Omega)}\preceq (h^{r+\beta}+h^r\varepsilon^s+\varepsilon^{1+s})\|f\|_{H^{s-1/2}(\Gamma)} .
\]
\end{remark}
\begin{remark}[choices of $\varepsilon$]\label{r:error-estimate}
Usually we can choose $\varepsilon = ch^q$ for some $q\in(0,1]$ and for a fixed factor $c$. Hence, error estimates in the above remark become
\[
	\|u-\ueh\|_{H^1(\Omega)}\preceq h^{\min\{\beta,{sq}\}}, 	\quad
	\|u-\ueh\|_{L^2(\Omega)}\preceq h^{\min\{\beta+r, r+sq, (1+s)q\}} .
\]
\end{remark}

\section{Numerical Illustrations}\label{s:numerical}
We implement the linear system of discrete problem \eqref{e:discrete}
using the \textit{deal.II} finite element Library \cite{MR3893339,
  MaierBardelloniHeltai-2016-a,
  SartoriGiulianiBardelloni-2018-a}. Before validating our error
estimates given by Theorem~\ref{t:discrete-h1-estimate} and
\ref{t:discrete-l2-estimate} via a series of numerical experiments, we
want to make some remarks on the computation of the right hand side
vector in discrete linear system.

\begin{remark}[Approximation of the surface $\Gamma$]\label{r:surface}
  We shall compute the right hand side data on an approximation of a
  $C^2$ interface $\Gamma$ using the technology of the surface finite
  element method for the Laplace-Beltrami problem
  \cite{MR976234,MR2485433}. Let $\Gamma_{h_0}$ be a polytope which
  consists of simplices with co-dimension one, where $h_0$ denotes the
  maximal size of the subdivision. All vertices of these simplices lie
  on $\Gamma$ and similar to $\mathcal T_h(\Omega)$, we assume that
  this subdivision of $\Gamma_{h_0}$, denoted by
  $\mathcal T_{h_0}(\Gamma)$, is conforming, shape-regular, and
  quasi-uniform. We compute the forcing data at $\Gamma_{h_0}$ with
  $f^e(x) = f(p(x))$, where $p(x)=x-d(x)\GRAD d(x)\in \Gamma$ and
  $d(x)$ is the signed distance function for $\Gamma$. The error
  analysis for how the finite element approximation of $u$ in
  \eqref{e:strong} is affected by $\Gamma_{h_0}$ is out of the scope
  of this paper. In what follows, we assume that $h_0$ is small enough
  compared with $h$ so that the error from the interface approximation
  does not affect the total error of our test problems.
\end{remark}

\begin{remark}[Comparing with the usual approach]\label{r:compute}
Based on the previous remark and given a shape function $\phi_h\in \mathbb V_h$, we can approximate the right hand side data $\langle\Fepsilon,\phi_h\rangle_{V',V}$ using quadrature rules on both $\tau_1\in\mathcal T_h(\Omega)$ and $\tau_2\in \mathcal T_{h_0}(\Gamma)$, \ie
\[
\begin{aligned}
	\langle\Fepsilon&,\phi_h\rangle_{V',V} \approx \int_\Omega \int_{\Gamma_{h_0}} \depsilon (x-y) {f^e(y) \phi_h(x)}\diff \sigma_y\diff x \\
	&\approx \sum_{\tau_1\in\mathcal T_h(\Omega)\cap \emph{supp}(\phi_h)}\sum_{\tau_2\in \mathcal T_{h_0}(\Gamma), \tau_2\cap \tau_1^\varepsilon\neq\O}\sum_{j_1=1}^{J_{\tau_1}}\sum_{j_2=1}^{J_{\tau_2}}
	w_{j_1}w_{j_2}\depsilon(q_{j_1}-q_{j_2}) f^e(q_{j_2}) \phi_h(q_{j_1}) .
\end{aligned}
\]
Here $\tau_1^\varepsilon$ follows from the definition \eqref{e:extension} and $\emph{supp}(\phi_h)$
denotes the support of $\phi_h$. $\{w_{j_1},q_{j_1}\}_{j_1=1}^{J_{\tau_1}}$ and
$\{w_{j_2},q_{j_2}\}_{j_2=1}^{J_{\tau_2}}$ are pairs of quadrature
weights and quadrature points defined on $\tau_1\in\mathcal
T_h(\Omega)$ and $\tau_2\in \mathcal T_{h_0}(\Gamma)$,
respectively. On the other hand, letting $Q$ be the collection of quadrature points for all $\tau_2\in \mathcal T_{h_0}(\Gamma)$, the finite element method allows one
to approximate directly $\langle F,v_h\rangle_{V',V}$ by
\[
\begin{aligned}
	\langle F,\phi_h\rangle_{V',V} &\approx  \int_{\Gamma_{h_0}} f^e(y) \phi_h(y)\diff \sigma_y 
	\approx \sum_{\tau_2\in \mathcal T_{h_0}(\Gamma)} \sum_{j_2=1}^{J_{\tau_2}} w_{j_2}  f^e(q_{j_2}) \phi_h(q_{j_2}) \\
	&= \sum_{\tau_1\in \mathcal T_h(\Omega)\cap \emph{supp}(\phi_h)} \sum_{q_{j_2}\in Q, q_{j_2}\in\tau_1}  w_{j_2}  f^e(q_{j_2}) \phi_h(q_{j_2}) .
\end{aligned}
\]
We note that we usually compute $\phi_h(q_{j_1})$ by using a
transformation $B_{\tau_1}$ mapping from the reference simplex
$\hat\tau$ to $\tau_1\in \mathcal T_h(\Omega)$ and set
$q_{j_1} = B_{\tau_1}\hat q_{j_1}$ and
$w_{j_1} = \hat w_{j_1}|\det B_{\tau_1}|$, where
$\{\hat w_{j_1},\hat q_{j_1}\}$ is the pair of quadrature weights and
quadrature points for $\hat\tau$. In the implementation, we store the
reference shape function $\hat\phi_h$, the shape values
$\hat\phi_h(\hat q_{j_1})=\phi_h(q_{j_1})$ and $w_{j_1}$ offline (and similar
implementations for $\tau_2$) but we have to compute
$ \phi_h(q_{j_2})$ online by $\hat\phi_h(B_{\tau_1}^{-1} q_{j_2})$.

We finally note that searching for  $q_{j_2}\in Q$ and $\tau_2\in \mathcal T_{h_0}(\Gamma)$ intersecting with $\tau_1$ can be accelerated by creating R-trees for $q_{j_2}$ and bounding boxes of $\tau_2$, respectively. Then we use the searching algorithms for intersection with the bounding box of $\tau_1\in \mathcal T_h(\Omega)$. In our implementation, we use the searching algorithms from the \emph{Boost.Geometry.Index} library.
\end{remark}

\subsection{Tests on the unit square domain with different $\depsilon(x)$}
We test the interface problem \eqref{e:strong} on the unit square
domain but with a non-homogeneous Dirichlet boundary condition,
following the examples provided in~\cite{HeltaiRotundo-2019-a}. We set
$\Omega=(0,1)^2$ and solve
\begin{equation}\label{e:test}
	\begin{aligned}
		-\Delta u = 0, \quad &\text{ in } \Omega\backslash\Gamma, \\
		\llbracket u \rrbracket = 0, \quad &\text{ on } \Gamma,\\
		\bigg\llbracket \frac{\partial u}{\partial \nu} \bigg\rrbracket = f, \quad &\text{ on } \Gamma,\\
		u = g, \quad &\text{ on } \partial\Omega ,
	\end{aligned}
\end{equation}
where $\Gamma=\partial B_{0.2}(\mathbf{c})$ with $\mathbf{c}=(0.3,0.3)^{\Tr}$, $f=\tfrac{1}{0.2}$ and $g=\ln(|x-\mathbf{c}|)$. The analytic solution is
\[
	u(x) = \left\{
	\begin{aligned}
	-\ln(|x-\mathbf{c}|),&\quad \text{if } |x-\mathbf{c}|>0.2,\\
	-\ln(0.2),&\quad \text{if } |x-\mathbf{c}|\le 0.2 .
	\end{aligned}
	\right. 
\]
In view of Assumption~\ref{a:elliptic-regularity} and Remark~\ref{r:error-estimate}, $r=1$, $s=\tfrac12$ and $u\in H^{3/2^-}(\Omega)$. Hence setting $\varepsilon = h$ yields
\[
	\|u-\ueh\|_{H^1(\Omega)}\preceq h^{1/2}\sim \#\text{DoFs}^{-0.25}\quad
	\|u-\ueh\|_{L^2(\Omega)}\preceq h^{3/2}\sim {\#\text{DoFs}^{-0.75}},
\]
where $\#\text{DoFs}$ stands for the number of degree of freedoms and we used the fact that $h\sim \#\text{DoFs}^{-1/d}$ for quasi-uniform meshes. We shall compute $\ueh$ on a sequence of unstructured,  quasi-uniform,  quadrilateral meshes: we start to compute $u^\varepsilon_{h_1}$ on a coarse mesh of $\Omega$ in Figure~\ref{f:sq-mesh-sol}. Then we compute the next approximated solution in a higher resolution based on the global refinement of the previous mesh. In the meantime, we also take the global refinement of approximated interface according to Remark~\ref{r:surface}. The right plot in Figure~\ref{f:sq-mesh-sol} shows the approximated solution on the mesh produced from the coarse mesh with six-time global refinement (744705 degree of freedoms). 

In Figure~\ref{f:sq-error} we report $L^2(\Omega)$ and $H^1(\Omega)$ errors against $\#\text{DoFs}$ using the following types of $\depsilon(x)$:
\begin{itemize}
\item \emph{Radially symmetric $C^1$}: use \eqref{e:delta-app-1} with $\psi_\rho(x) = (1+\cos(|\pi x|))\chi_{B_1(0)}(x)/2$, where $\chi_{B_1(0)}(x)$ is the characteristic function on $B_1(0)$;
\item \emph{Tensor product $C^1$}: use \eqref{e:delta-app-2} with $\psi_{1d}(x) = (1+\cos(|\pi x|))\chi_{(-1,1)}(x)/2$;
\item \emph{Tensor product $C^\infty$}: use \eqref{e:delta-app-2} with $\psi_{1d} (x) = e^{1-1/(1-|x|^2)}\chi_{(-1,1)}(x)$;
\item \emph{Tensor product $L^\infty$}: use \eqref{e:delta-app-2} with $\psi_{1d} (x) = \tfrac12\chi_{(-1,1)}(x)$.
\end{itemize}
\begin{figure}[hbt!]
\begin{center}
\begin{tabular}{cc}
\!\!\!\!\!\!\!\!\!\!\!\!\!\!\!\includegraphics[scale=0.23]{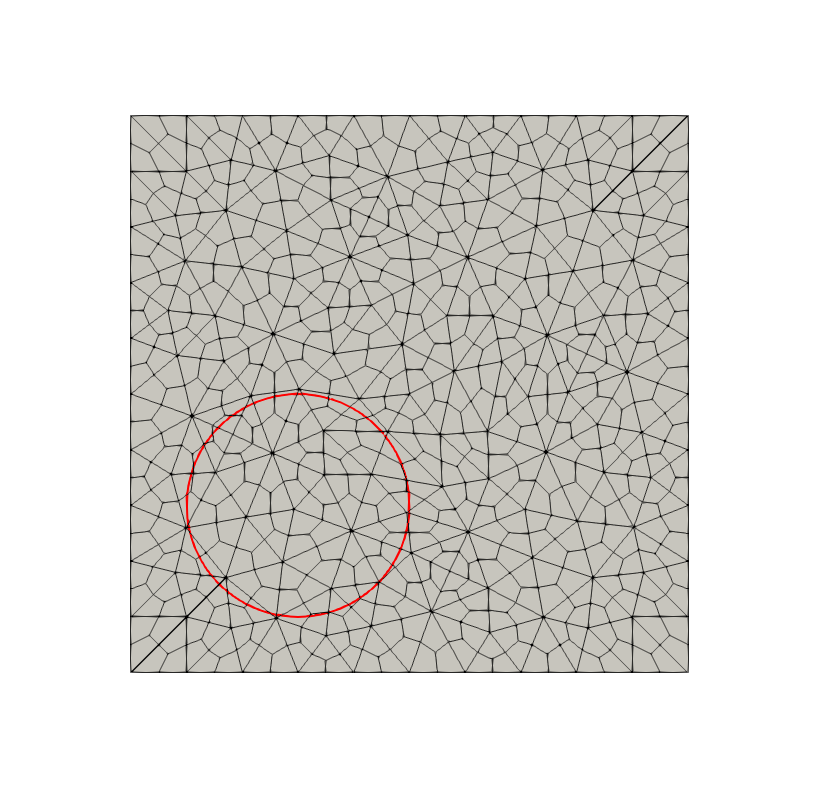} 
&\!\!\!\!\!\!\!\!\!\!\!\!\!\!\! \includegraphics[scale=0.37]{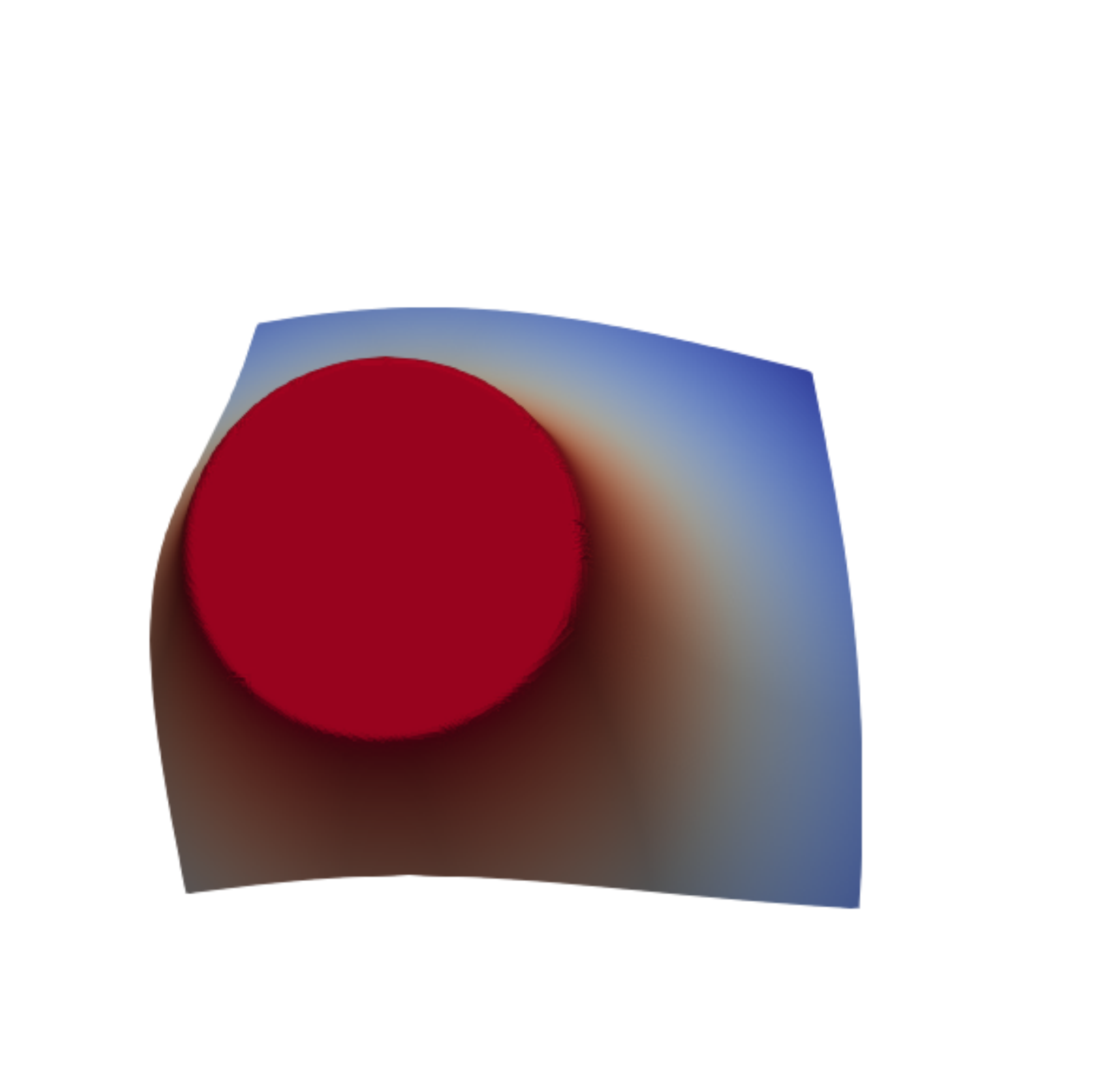} \\
\end{tabular}
\end{center}
\caption{(Left) the coarse meshes of the unit square domain $\Omega$ in black and the interface $\Gamma$ in red and (right) the approximated solution on the six-time-global-refinement mesh using the regularization type \emph{Tensor product $C^1$} and $\varepsilon = h$.}
\label{f:sq-mesh-sol}
\end{figure}
The predicted rates are observed in all four types. Errors in the first three types behave similar while errors for the last type are larger than those in the other cases.
\begin{figure}[hbt!]
\begin{center}
\includegraphics[scale=0.5]{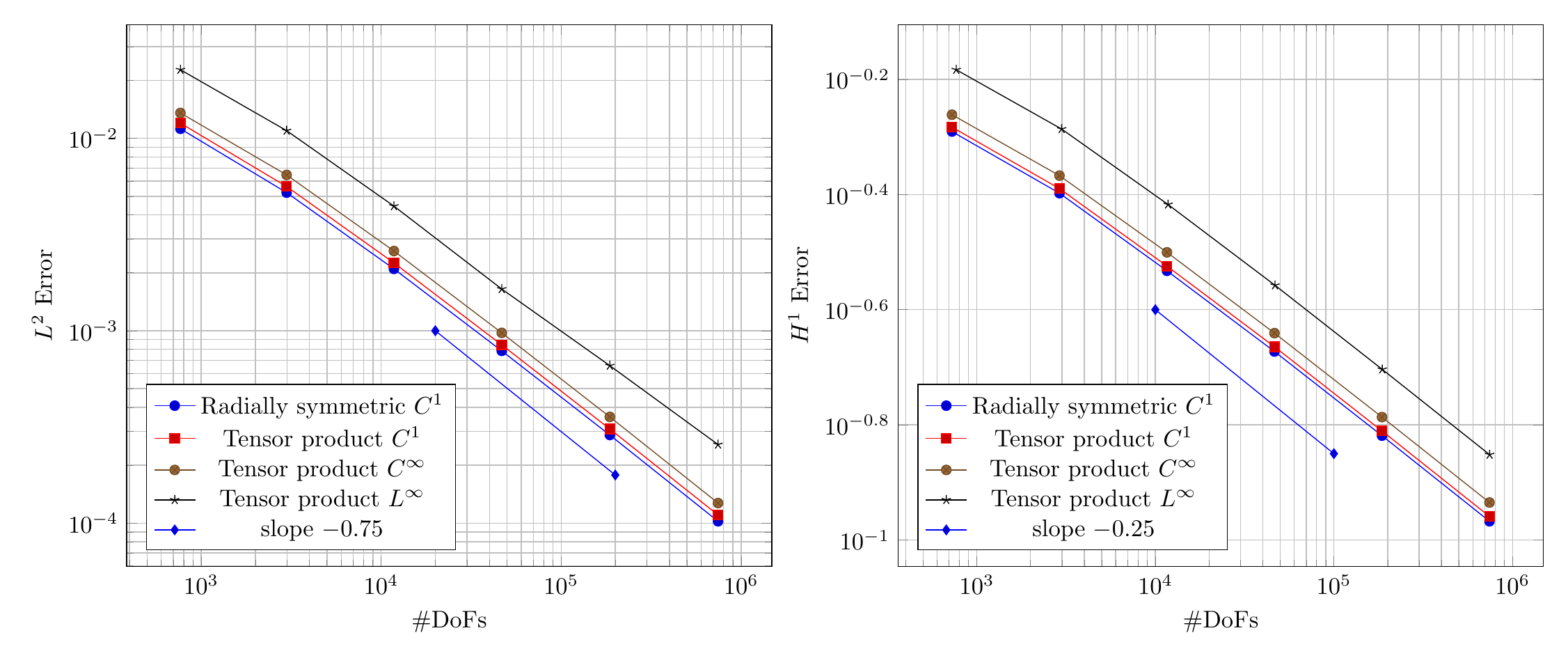} 
\end{center}
\caption{$L^2(\Omega)$ and $H^1(\Omega)$ errors against the number of degree of freedoms ($\#$DoFs) using different types of $\depsilon$.}
\label{f:sq-error}
\end{figure}

\subsection{Tests on a L-shaped domain}
Let $\Omega$ be the L-shaped domain {$(-1,1)^2\backslash ([0,1]\times[-1,0])$} and $\Gamma=\partial B_{0.2}(\mathbf{c})$ with $\mathbf{c}=(-0.5,-0.5)^{\Tr}$. We use the polar coordinates $(r(x),\theta(x))$ to define $u$ by
\[
	u(x)=r(x)^{\tfrac13}\sin(\tfrac{\theta(x)}3)+0.3
	\left\{
	\begin{aligned}
	-\ln(|x-\mathbf{c}|),&\quad \text{if } |x-\mathbf{c}|>0.2,\\
	-\ln(0.2),&\quad \text{if } |x-\mathbf{c}|\le 0.2 
	\end{aligned}
	\right.
\]
so that $u\in H^{4/3^-}(\Omega)$ is the solution to the test problem \eqref{e:test} with $f=1.5$ on $\Gamma$ and $g=u$ on $\partial\Omega$. Assumption~\ref{a:elliptic-regularity} and Remark~\ref{r:error-estimate} imply that $r=\tfrac23$, $s=\tfrac12$ and $\beta=\tfrac13$. So letting $ \varepsilon=h^q, q\in(0,1]$, we should expect that
\[
	\|u-\ueh\|_{H^1(\Omega)}\preceq \#\text{DoFs}^{-\min\{1/6,q/4\}}
\]
and
\[
	\|u-\ueh\|_{L^2(\Omega)}\preceq  \#\text{DoFs}^{-\min\{1/2,1/3+q/4,3q/4\}}.
\]
Figure~\ref{f:lshaped-mesh-sol} provides the coarse mesh of the numerical test and also the solution on the mesh with six-time global refinement (1574913 degrees of freedoms). In Figure~\ref{f:lshaped-error} we also report $L^2(\Omega)$ and $H^1(\Omega)$ errors against $ \#\text{DoFs}$ with $q=0.2,0.4,0.6,0.8,1$ using the $\depsilon(x)$ with the type \emph{Tensor product $C^1$}. We again observed the predicted convergence rates from Figure~\ref{f:lshaped-error}.
\begin{figure}[hbt!]
\begin{center}
\begin{tabular}{cc}
\!\!\!\!\!\!\!\!\!\!\!\!\!\!\!\includegraphics[scale=0.24]{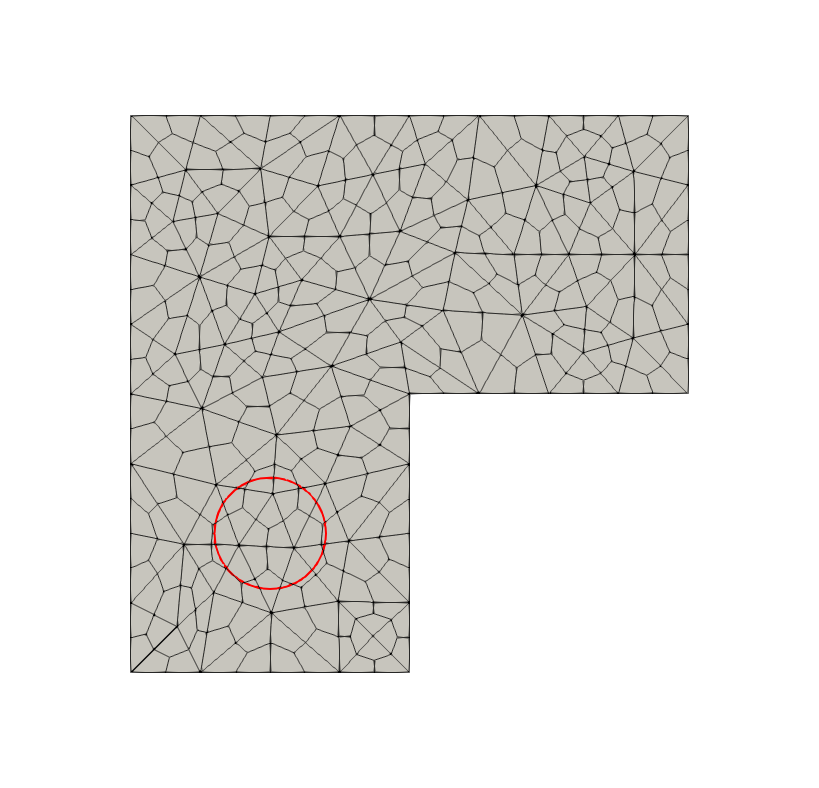} 
& \!\!\!\!\!\!\!\!\!\!\!\!\!\!\!\includegraphics[scale=0.11]{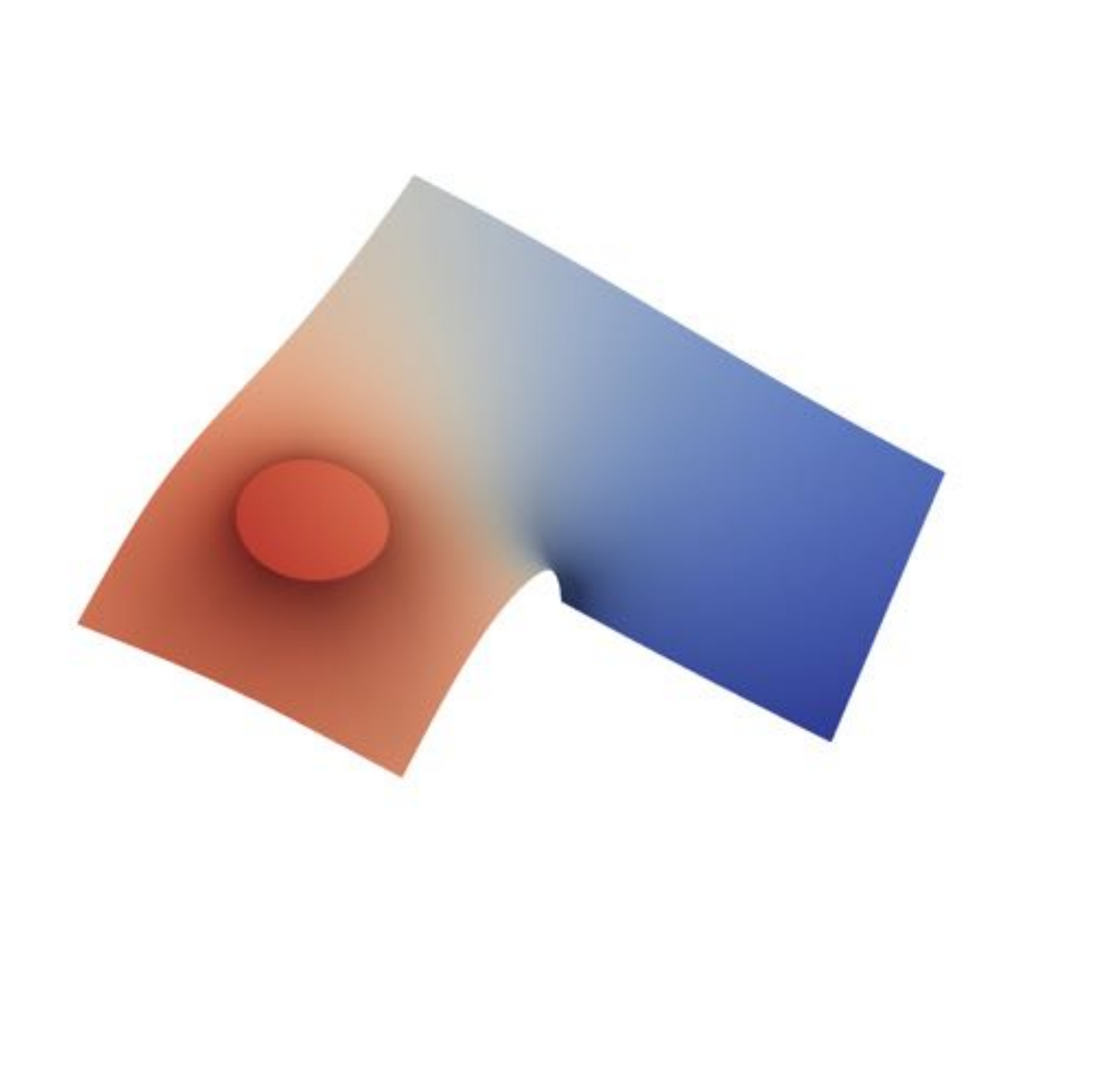} \\
\end{tabular}
\end{center}
\caption{(Left) the coarse meshes of the L-shaped domain $\Omega$ in black and the interface $\Gamma$ in red and (right) the approximated solution on the six-time-global-refinement mesh.}
\label{f:lshaped-mesh-sol}
\end{figure}
\begin{figure}[hbt!]
\begin{center}
\includegraphics[scale=0.5]{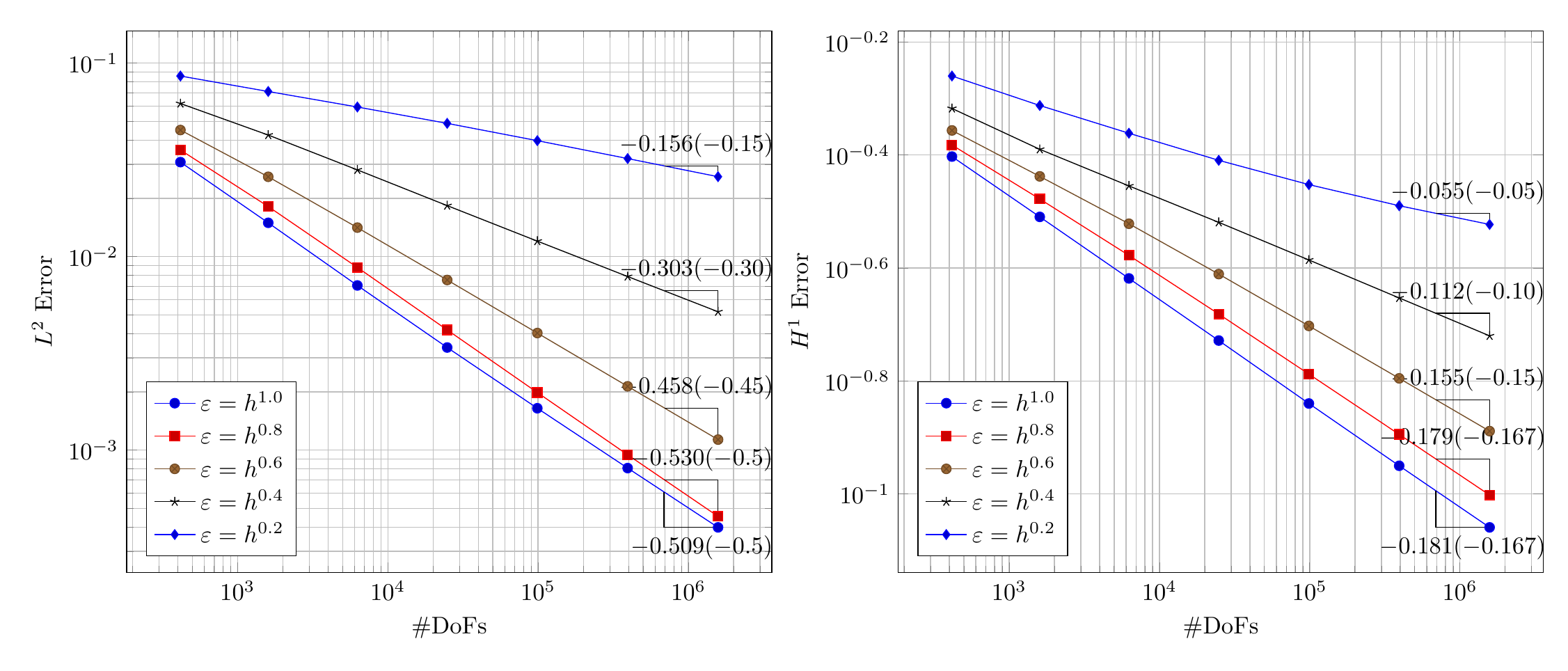}
\end{center}
\caption{$L^2(\Omega)$ and $H^1(\Omega)$ errors against the number of degree of freedoms ($\#$DoFs) using the type \emph{Tensor product $C^1$} for $\depsilon$. For each fixed type of $\depsilon$ we plot the error decay with the setting $\varepsilon = h^q$ for $q=0.2,0.4,0.6,0.8,1$. For each error plot, the last slope of the segment together with the predicted convergence rate are reported.}
\label{f:lshaped-error}
\end{figure}

\subsection{Tests on the unit cube}
We finally test the interface problem \eqref{e:test} in the three dimensional space by setting $\Omega=(0,1)^3$, $\Gamma=B_{0.2}(\mathbf{c})$ with $\mathbf{c}=(0.3,0.3,0.3)^{\Tr}$, $f=1/0.2^2$ and $g=1/|x-\mathbf{c}|$. Then the analytical solution is given by
\[
	u(x) = \left\{
	\begin{aligned}
	\frac1{|x-\mathbf{c}|},&\quad \text{if } |x-\mathbf{c}|>0.2,\\
	\frac1{0.2},&\quad \text{if } |x-\mathbf{c}|\le 0.2 .
	\end{aligned}
	\right. 
\]
Figure~\ref{f:cube-mesh-sol-unstructured} shows the unstructured coarse mesh of the unit cube as well as the approximated solution on the mesh after the forth-time global refinement (2324113 degrees of freedoms). In Figure~\ref{f:cube-error-time-unstructured} we plot the $L^2(\Omega)$ error decay by setting $\varepsilon=h$ and the error decay without using Dirac delta approximations as mentioned in Remark~\ref{r:compute}. We also plot the CPU time for the computation of the right hand side vector against $\#\text{DoFs}$ with or without regularization in Figure~\ref{f:cube-error-time-unstructured}. We note that the the computer we use has 2.2 GHz Intel Core i7 with 16GB memory.  Figure~\ref{f:cube-error-time-unstructured} shows that both the computation time for the right hand side assembling and error using the regularization approach are comparable to those using the usual method. We have to point out that according to Remark~\ref{r:compute}, we cannot evaluate $B_{\tau_1}^{-1}$ explicitly when the mesh is unstructured and here we use the Newton iteration instead. So if the geometry of each element is simple such as cube, the computation time without regularization can be reduced significantly.
\begin{figure}[hbt!]
\begin{center}
\begin{tabular}{cc}
\includegraphics[scale=0.09]{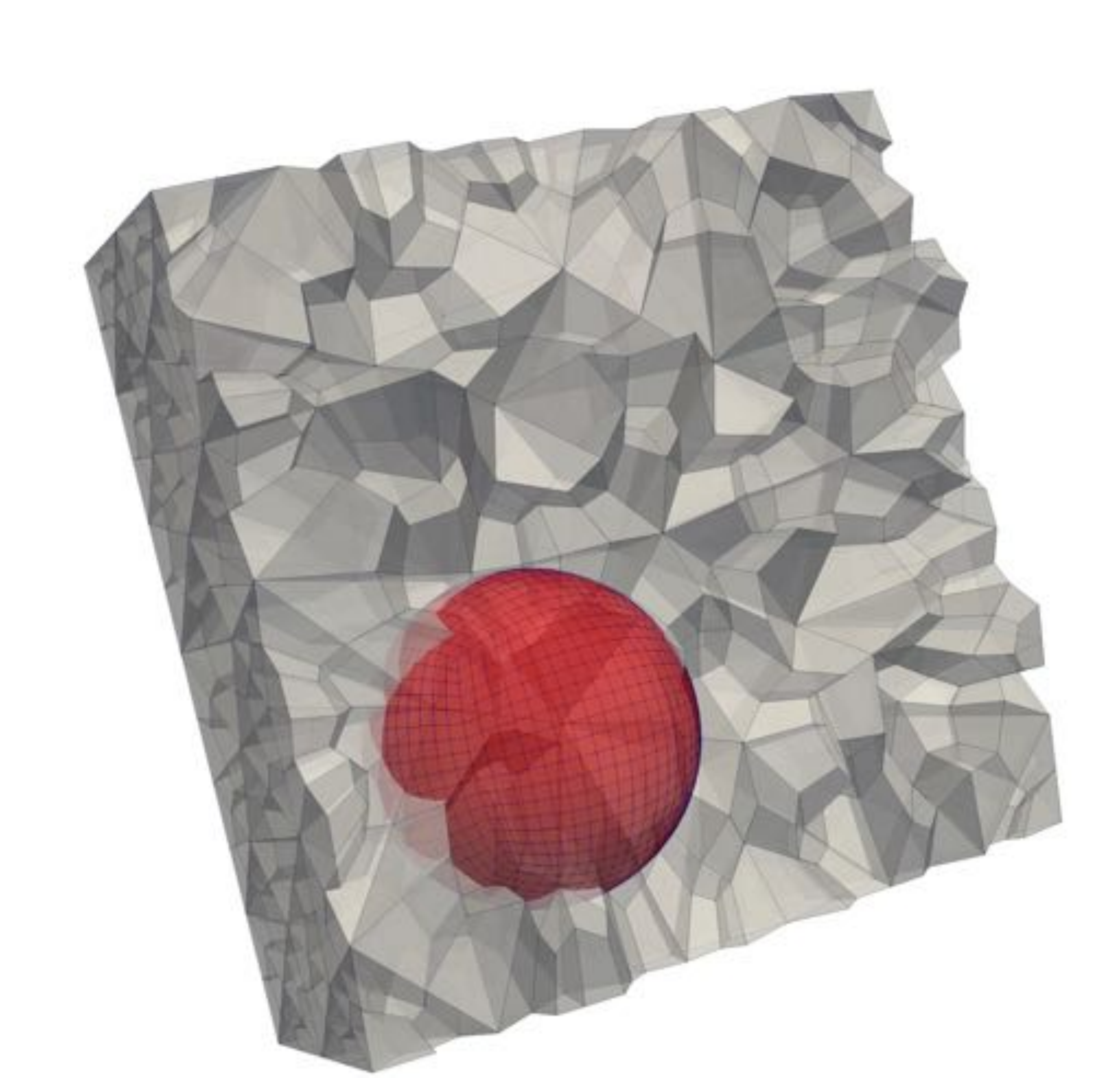} 
& \includegraphics[scale=0.1]{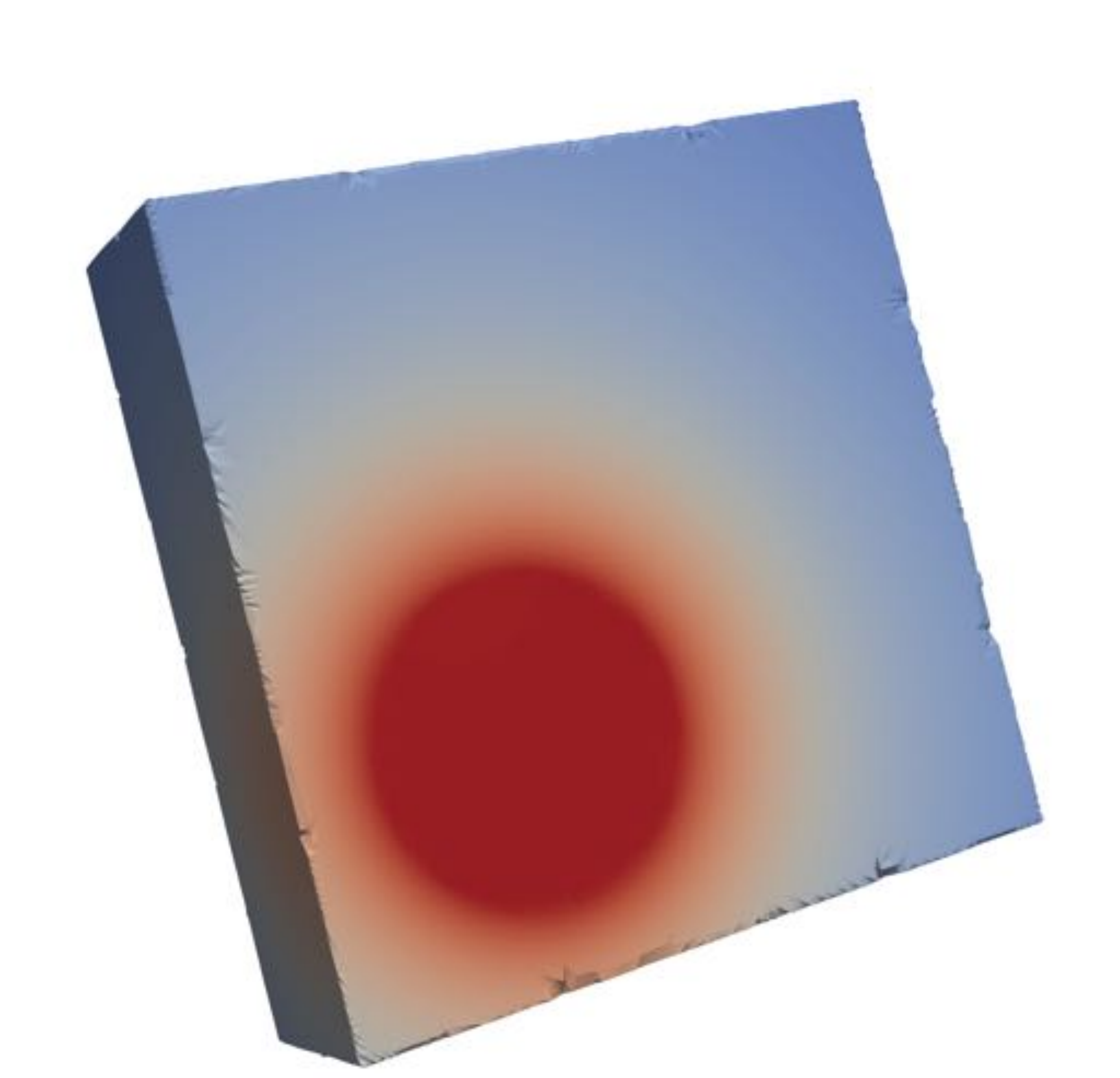} \\
\end{tabular}
\end{center}
\caption{(Left) clip of the coarse mesh ($0\le x_1\le 0.3$) of the unit cube $\Omega$ together with the subdivision of the interface $\Gamma$ in red and (right) the approximated solution computed on the three-time global refinement mesh using the regularization with the type \emph{Tensor product $C^1$}.}
\label{f:cube-mesh-sol-unstructured}
\end{figure}
\begin{figure}[hbt!]
\begin{center}
\includegraphics[scale=0.5]{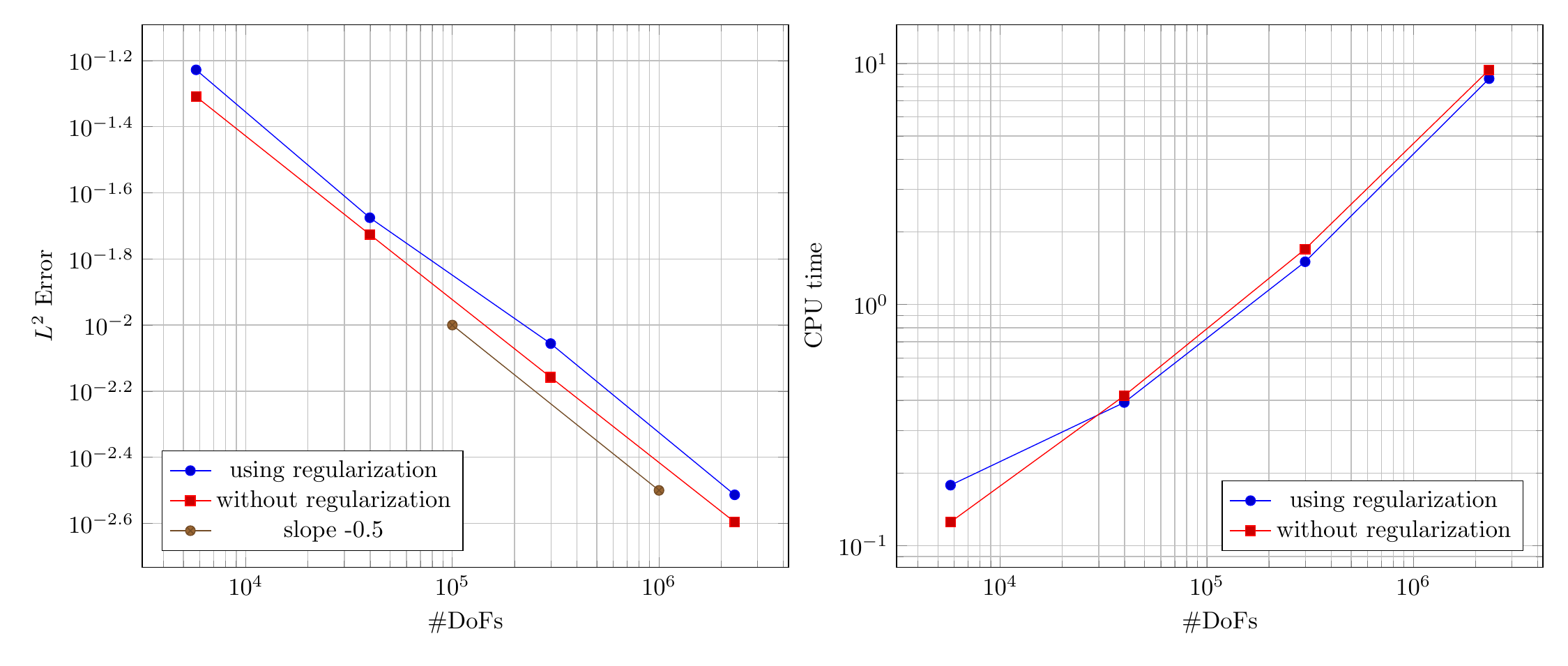} 
\end{center}
\caption{Plots of $L^2(\Omega)$ error decay (left) and CPU time for right hand side assembling against the number of degree of freedoms using with or without regularization according to Remark~\ref{r:compute}. We regularize the forcing data using the type \emph{Tensor product C1} with $\varepsilon=h$.}
\label{f:cube-error-time-unstructured}
\end{figure}
\bibliographystyle{plain}

\begin{thebibliography}{10}
\providecommand{\url}[1]{{#1}}
\providecommand{\urlprefix}{URL }
\expandafter\ifx\csname urlstyle\endcsname\relax
  \providecommand{\doi}[1]{DOI~\discretionary{}{}{}#1}\else
  \providecommand{\doi}{DOI~\discretionary{}{}{}\begingroup
  \urlstyle{rm}\Url}\fi

\bibitem{MR3893339}
Alzetta, G., Arndt, D., Bangerth, W., Boddu, V., Brands, B., Davydov, D.,
  Gassm\"{o}ller, R., Heister, T., Heltai, L., Kormann, K., Kronbichler, M.,
  Maier, M., Pelteret, J.P., Turcksin, B., Wells, D.: The deal.{II} library,
  version 9.0.
\newblock J. Numer. Math. \textbf{26}(4), 173--183 (2018)

\bibitem{AlzettaHeltai-2020-a}
Alzetta, G., Heltai, L.: Multiscale modeling of fiber reinforced materials via
  non-matching immersed methods.
\newblock Computers {\&} Structures \textbf{239}, 106334 (2020)

\bibitem{BoffiGastaldi-2003-a}
Boffi, D., Gastaldi, L.: {A finite element approach for the immersed boundary
  method}.
\newblock Computers {\&} Structures \textbf{81}(8-11), 491--501 (2003)

\bibitem{BoffiGastaldiHeltai-2008-a}
Boffi, D., Gastaldi, L., Heltai, L., Peskin, C.S.: {On the hyper-elastic
  formulation of the immersed boundary method}.
\newblock Computer Methods in Applied Mechanics and Engineering
  \textbf{197}(25-28), 2210--2231 (2008)

\bibitem{MR1930132}
Ciarlet, P.G.: The finite element method for elliptic problems, \emph{Classics
  in Applied Mathematics}, vol.~40.
\newblock Society for Industrial and Applied Mathematics (SIAM), Philadelphia,
  PA (2002).
\newblock Reprint of the 1978 original [North-Holland, Amsterdam; MR0520174 (58
  \#25001)]

\bibitem{MR961439}
Dauge, M.: Elliptic boundary value problems on corner domains, \emph{Lecture
  Notes in Mathematics}, vol. 1341.
\newblock Springer-Verlag, Berlin (1988).
\newblock Smoothness and asymptotics of solutions

\bibitem{MR2895178}
Demengel, F., Demengel, G.: Functional spaces for the theory of elliptic
  partial differential equations.
\newblock Universitext. Springer, London; EDP Sciences, Les Ulis (2012).
\newblock Translated from the 2007 French original by Reinie Ern\'{e}

\bibitem{MR2485433}
Demlow, A.: Higher-order finite element methods and pointwise error estimates
  for elliptic problems on surfaces.
\newblock SIAM J. Numer. Anal. \textbf{47}(2), 805--827 (2009)

\bibitem{MR976234}
Dziuk, G.: Finite elements for the {B}eltrami operator on arbitrary surfaces.
\newblock In: Partial differential equations and calculus of variations,
  \emph{Lecture Notes in Math.}, vol. 1357, pp. 142--155. Springer, Berlin
  (1988)

\bibitem{MR2050138}
Ern, A., Guermond, J.L.: Theory and practice of finite elements, \emph{Applied
  Mathematical Sciences}, vol. 159.
\newblock Springer-Verlag, New York (2004)

\bibitem{MR2597943}
Evans, L.C.: Partial differential equations, \emph{Graduate Studies in
  Mathematics}, vol.~19, second edn.
\newblock American Mathematical Society, Providence, RI (2010)

\bibitem{Friedrichs1944}
Friedrichs, K.O.: The identity of weak and strong extensions of differential
  operators.
\newblock Transactions of the American Mathematical Society \textbf{55},
  132--132 (1944)

\bibitem{GlowinskiPanPeriaux-1994-a}
Glowinski, R., Pan, T.W., P{\'{e}}riaux, J.: {A fictitious domain method for
  external incompressible viscous flow modeled by {N}avier-{S}tokes equations}.
\newblock Comput. Methods Appl. Mech. Engrg. \textbf{112}(1-4), 133--148 (1994)

\bibitem{GriffithPeskin-2005-a}
Griffith, B.E., Peskin, C.S.: {On the order of accuracy of the immersed
  boundary method: Higher order convergence rates for sufficiently smooth
  problems}.
\newblock Journal of Computational Physics \textbf{208}(1), 75--105 (2005)

\bibitem{Heltai-2008-a}
Heltai, L.: On the stability of the finite element immersed boundary method.
\newblock Computers \& Structures \textbf{86}(7-8), 598--617 (2008)

\bibitem{HeltaiCaiazzo-2018-a}
Heltai, L., Caiazzo, A.: Multiscale modeling of vascularized tissues via
  non-matching immersed methods.
\newblock International Journal for Numerical Methods in Biomedical Engineering
  \textbf{35}(12), e3264 (2019)

\bibitem{HeltaiCostanzo-2012-a}
Heltai, L., Costanzo, F.: {Variational implementation of immersed finite
  element methods}.
\newblock Computer Methods in Applied Mechanics and Engineering
  \textbf{229-232} (2012)

\bibitem{HeltaiRotundo-2019-a}
Heltai, L., Rotundo, N.: Error estimates in weighted sobolev norms for finite
  element immersed interface methods.
\newblock Computers {\&} Mathematics with Applications \textbf{78}(11),
  3586--3604 (2019)

\bibitem{MR3429589}
Hosseini, B., Nigam, N., Stockie, J.M.: On regularizations of the {D}irac delta
  distribution.
\newblock J. Comput. Phys. \textbf{305}, 423--447 (2016)

\bibitem{MR2441884}
Hsiao, G.C., Wendland, W.L.: Boundary integral equations, \emph{Applied
  Mathematical Sciences}, vol. 164.
\newblock Springer-Verlag, Berlin (2008)

\bibitem{LaiPeskin-2000-a}
Lai, M.C., Peskin, C.S.: {An Immersed Boundary Method with Formal Second-Order
  Accuracy and Reduced Numerical Viscosity}.
\newblock Journal of Computational Physics \textbf{160}(2), 705--719 (2000)

\bibitem{LevequeLi-1994-a}
Leveque, R.J.., Li, Z.: {The Immersed Interface Method for Elliptic Equations
  with Discontinuous Coefficients and Singular Sources}.
\newblock SIAM J. Numer. Anal. \textbf{31}(4), 1019--1044 (1994)

\bibitem{LiuMori-2012-a}
Liu, Y., Mori, Y.: {Properties of Discrete Delta Functions and Local
  Convergence of the Immersed Boundary Method}.
\newblock SIAM Journal on Numerical Analysis \textbf{50}(6), 2986--3015 (2012)

\bibitem{LiuMori-2014-a}
Liu, Y., Mori, Y.: {{\$}L{\^{}}p{\$} Convergence of the Immersed Boundary
  Method for Stationary Stokes Problems}.
\newblock SIAM Journal on Numerical Analysis \textbf{52}(1), 496--514 (2014)

\bibitem{MaierBardelloniHeltai-2016-a}
Maier, M., Bardelloni, M., Heltai, L.: {\texttt{LinearOperator} -- a generic,
  high-level expression syntax for linear algebra}.
\newblock Computers and Mathematics with Applications \textbf{72}(1), 1--24
  (2016)

\bibitem{MR1742312}
McLean, W.: Strongly elliptic systems and boundary integral equations.
\newblock Cambridge University Press, Cambridge (2000)

\bibitem{MittalIaccarino-2005-a}
Mittal, R., Iaccarino, G.: {Immersed boundary methods}.
\newblock Annual Review of Fluid Mechanics \textbf{37}(1), 239--261 (2005)

\bibitem{Mori-2008-a}
Mori, Y.: {Convergence proof of the velocity field for a stokes flow immersed
  boundary method}.
\newblock Communications on Pure and Applied Mathematics \textbf{61}(9),
  1213--1263 (2008)

\bibitem{Peskin-2002-a}
Peskin, C.S.: {The immersed boundary method}.
\newblock Acta Numerica \textbf{11}(1), 479--517 (2002)

\bibitem{SaitoSugitani-2019-a}
Saito, N., Sugitani, Y.: {Analysis of the immersed boundary method for a finite
  element Stokes problem}.
\newblock Numerical Methods for Partial Differential Equations \textbf{35}(1),
  181--199 (2019)

\bibitem{SartoriGiulianiBardelloni-2018-a}
Sartori, A., Giuliani, N., Bardelloni, M., Heltai, L.: {deal2lkit: A toolkit
  library for high performance programming in deal.II}.
\newblock SoftwareX \textbf{7}, 318--327 (2018)

\bibitem{MR1011446}
Scott, L.R., Zhang, S.: Finite element interpolation of nonsmooth functions
  satisfying boundary conditions.
\newblock Math. Comp. \textbf{54}(190), 483--493 (1990)

\bibitem{soboleff1938theoreme}
Soboleff, S.: Sur un th{\'e}or{\`e}me d'analyse fonctionnelle.
\newblock Matematicheskii Sbornik \textbf{46}(3), 471--497 (1938)

\bibitem{TornbergEngquist-2004-a}
Tornberg, A.K., Engquist, B.: {Numerical approximations of singular source
  terms in differential equations}.
\newblock Journal of Computational Physics \textbf{200}(2), 462--488 (2004)

\bibitem{Young1912}
Young, W.H.: On the multiplication of successions of fourier constants.
\newblock Proceedings of the Royal Society A: Mathematical, Physical and
  Engineering Sciences \textbf{87}(596), 331--339 (1912)

\end{thebibliography}

\end{document}